\documentclass[a4paper, english, 12pt, reqno, dvipsnames, final]{amsart}

\usepackage{amssymb}
\usepackage{a4wide}

\usepackage[english]{babel}
\usepackage[utf8]{inputenc}
\usepackage{multicol}
\usepackage{graphicx}
\usepackage{xcolor}
\usepackage{enumerate,enumitem}
\usepackage{booktabs,multirow}
\usepackage{amsfonts}
\usepackage{dsfont}
\usepackage{mathtools}
\usepackage{todonotes}
\usepackage{tikz,tikzscale}
\usepackage{pgfplots}
\usetikzlibrary{arrows.meta, angles, quotes}

\newtheorem{theorem}{Theorem}[section]
\newtheorem{lemma}[theorem]{Lemma}

\theoremstyle{definition}

\theoremstyle{remark}
\newtheorem{remark}[theorem]{Remark}

\numberwithin{equation}{section}

\let\rho\varrho

\newcommand{\Edge}{\ensuremath{E}}

\newcommand{\Node}{\ensuremath{N}}

\newcommand{\setEdge}{\ensuremath{\mathcal E_h}}
\newcommand{\setNode}{\ensuremath{\mathcal N_h}}

\newcommand{\edge}{\ensuremath{\Edge}}
\newcommand{\node}{\ensuremath{\Node}}

\newcommand{\trace}{\ensuremath{\gamma}}

\renewcommand{\div}{\ensuremath{\nabla\!\cdot\!}}

\newcommand{\Normal}{\mathbf n}

\newcommand{\IR}{\ensuremath{\mathbb R}}
\newcommand{\spaceH}{\ensuremath{\mathcal H}}
\newcommand{\spaceM}{\ensuremath{\mathcal M}}
\newcommand{\Hdiv}{\mathbf H^{\textup{div}}}

\newcommand{\polynomials}{\ensuremath{Q}}
\newcommand{\polynomialsp}{\ensuremath{P}}
\newcommand{\polynomialsRT}{\ensuremath{RT}}

\renewcommand{\vec}[1]{\mathbf{#1}}
\DeclareMathAlphabet{\mathbfsf}{\encodingdefault}{\sfdefault}{bx}{n}

\newcommand{\dx}{\ensuremath{\, \textup d x}}

\newcommand{\ds}{\ensuremath{\, \textup d \sigma}}
\newcommand{\dbx}{\ensuremath{\, \textup d\vec x}}

\usepackage[colorlinks = true, linkcolor = blue, citecolor = blue, urlcolor = blue]{hyperref}

\makeatletter
\newcommand\footnoteref[1]{\protected@xdef\@thefnmark{\ref{#1}}\@footnotemark}
\makeatother

\usepackage[backend=bibtex,doi=false,isbn=false, url=false,
eprint=false, giveninits=true,
style=numeric,defernumbers=false,sorting=none,style=numeric-comp,maxbibnames=99]{biblatex}
\begin{document}

\title[On positivity preservation of HDG methods]{On positivity preservation of hybrid discontinuous Galerkin methods on hypergraphs} 

\author{Petr Knobloch}
\address{Department of Numerical Mathematics, Faculty of Mathematics and
Physics, Charles University, Prague, Czech Republic}
\email{knobloch@karlin.mff.cuni.cz}

\author{Philip L.\ Lederer}
\address{Department of Mathematics, University of Hamburg, Germany}
\email{philip.lederer@uni-hamburg.de}

\author{Andreas Rupp}
\address{Department of Mathematics, Faculty of Mathematics and Computer Science, Saarland University, Saarbrücken, Germany}
\email{andreas.rupp@uni-saarland.de}

% \thanks{}

% \subjclass[2010]{\textcolor{red}{TODO}}

\begin{abstract}
Hybrid finite element methods, particularly hybridized discontinuous Galerkin (HDG) methods, are efficient numerical schemes for discretizing the diffusion equation, which encompasses two main physical principles: mass conservation and positivity preservation. While the former has been extensively analyzed in the literature, this paper investigates the latter. We state a theorem that guarantees the positivity of both the bulk and skeleton approximations to the primary unknown (concentration) and provide counterexamples for nonpositive discretizations. The theoretical findings are confirmed by numerical experiments.
 \\[1ex] \noindent \textsc{Keywords.}
 Positivity preservation, hybridized discontinuous Galerkin methods, diffusion equation, mass conservation, numerical analysis, Raviart-Thomas
\end{abstract}

\date{\today}
\maketitle
\section{Introduction}
Structure preserving discretizations of the diffusion equation need to be locally mass conservative and, in the case of 
nonnegative data, they should also preserve the positivity of the solution. This is particularly important in applications where the solution represents a physical quantity, such as temperature or concentration, which is nonnegative by definition. While there is plenty of literature on positivity preserving (continuous) finite element schemes for diffusion problems \cite{BarrenecheaGPV23,BBK17,BJK17,BarrenecheaJK24,BrandtsKK08,Ciarlet70,CiarletR73}, these finite elements usually come with a reduced notion of local mass conservation---the local fluxes can only be reconstructed in a fully discrete fashion \cite{Abgrall23,HughesEML00}. For comprehensive overviews over such methods, we refer to \cite{KeithS24,BarrenecheaJK24,BarrenecheaJK25}. In contrast, discontinuous Galerkin (DG) and enriched Galerkin (EG) methods allow for an enhanced notion of local conservation with their naturally defined inter- and intra-element fluxes, which can be postprocessed to obtain a locally conservative solution concerning that flow field. However, the literature on locally bound preserving DG and EG schemes for the diffusion equation is significantly more scarce and mainly addresses parabolic equations, see \cite{FrankRK20}, while many approaches in the hyperbolic world exist \cite{KuzminHR20,Kuzmin21,KuzminLY24,Hajduk21,Pazner21,ErnG22,GuermondBT19}. We refer to \cite{KuzminH24} for an extensive overview. Some of these approaches, such as slope limiters \cite{Krivodonova07,Kuzmin10,Kuzmin20}, can be transferred to diffusion problems.

The most direct approach to obtain local mass conservation for the diffusion
equation comprises dual, mixed methods, which discretize the flux explicitly.
However, these methods generate a linear system of equations with a
saddle-point structure and the authors are unaware of any results on positivity
preservation. A way to bypass the saddle-point system lies in the concept of
hybridization, which has been generalized to cover a large class of so-called
hybrid discontinuous Galerkin (HDG) methods in \cite{CockburnGL09}. These
methods are locally mass conservative not only for classical domains but even
if they discretize the diffusion equation posed on graphs or networks of
surfaces, which naturally appear as singular limits of thin structures \cite{RuppGK22}.

The positivity preservation of such HDG methods has not been fully established,
which is the focus of this contribution. Nonetheless, there are several
works that consider the maximum principles and positivity preservation of
hybrid (mixed) methods, which can be considered as one class of HDG methods in
which no stabilization is necessary since the respective approximation spaces ensure an inf-sup stable discretization. For example, \cite{HoteitMPAE03,YounesAL06,YounesAAB10} consider parabolic equations. In particular, \cite{HoteitMPAE03} exploits a time step restriction of the form \( (\Delta x)^2 / \Delta t  \le C \) to obtain positivity preservation, while \cite{YounesAL06} uses a mass-lumping approach ensuring positivity preservation if all angles are less than or equal to \(\tfrac\pi2 \) and the spatial dimension is two. Moreover, \cite{YounesAAB10} divides distorted quadrilateral elements into triangles to obtain an accurate and positivity preserving approximation in two spatial dimensions.

Elliptic problems discretized by hybrid mixed finite elements are investigated in \cite{BreraJMS10,Mazzia08,MichelettiSS01} for two-dimensional domains. While \cite{Mazzia08} considers a mesh consisting of a maximum of two elements with nonobtuse angles and proves that the lowest order Raviart--Thomas discretization is positivity preserving in that case (giving rise to the hope that this is also true in general), \cite{BreraJMS10} considers a mesh of nonobtuse triangles and ensures maximum principles even for heterogeneous equations (where the diffusion coefficient can jump among subdomains which need to be exactly resolved---a specialized quadrature is necessary here). Finally, \cite{MichelettiSS01} demonstrates that for Delaunay-type triangulations, the relations between cell-centered finite volumes and lowest order (hybrid) mixed methods can be exploited to generate a $RT_0$ method with a special quadrature that guarantees maximum principles.

Positivity for the stationary advection--diffusion--reaction equation discretized by weak Galerkin methods, which are almost identical to HDG, is analyzed in \cite{LiuW20,WangXQR18,ZhouWJ22} for domains \( \Omega \subset \IR^2 \). In particular, \cite{LiuW20} considers rectangular meshes and observes positivity preservation if the elements' aspect ratios are in \( [\tfrac12, 2] \). In \cite{WangXQR18}, a weak Galerkin method with constant skeleton unknowns and linear bulk unknowns is shown to have a discrete maximum principle if an \( h \)-acute angle condition holds, which is a nonobtuse angle condition for the diffusion equation. Another weak Galerkin method, where the trace unknowns are face-wise linears and overall continuous while the bulk unknowns are broken linears is shown to satisfy a discrete maximum principle if the mesh satisfies a nonobtuse angle condition for simplicial meshes and an aspect ratio condition for rectangular meshes (if the diffusion equation is considered) in \cite{ZhouWJ22}. Considering the diffusion equation, we note that this scheme is very similar to the embedded discontinuous Galerkin method whose bulk unknown can be shown to be interpolation of the trace unknown (if the diffusion coefficient is an element-wise constant scalar). Thus, we expect this method to have similar properties (concerning discretization matrix, positivity, and convergence as continuous linear finite elements). The anisotropic diffusion equation is analyzed in \cite{HuangW15}, where the authors prove the anisotropic nonobtuse angle condition to be sufficient for positivity preservation. For isotropic tensors, this condition coincides with the non-obtuse angle condition.

As in \cite{LiuW20,WangXQR18,ZhouWJ22} we consider hybrid and discontinuous methods on simplicial and quadrilateral meshes. However, we do not restrict ourselves to two spatial dimensions, allow elements with obtuse angles, and also consider the lowest order case (where all polynomials are local constants). Moreover, our positivity preservation results hold in a point-wise manner (and not only at quadrature points) and we do not require specific quadrature rules, in contrast to many of the aforementioned papers. This is particularly important since the result of the diffusion problem may be used as data for other numerical simulations. We observe that our methods remain positivity preserving under similar conditions as the weak Galerkin method, but if those conditions are violated, we are also able to ensure positivity by over-penalizing the HDG method and trading in convergence (order) for positivity preservation. This feature is useful if HDG schemes are used as positivity-preserving low-order method in a flux-corrected/limited transport scheme, which we aim to discuss in future works. Moreover, we explore the sharpness of the angle and stabilization conditions by providing counterexamples. We expect that the above results can be transferreed to hybrid high-order methods, which are closely linked to HDG, see \cite{CockburnDE16}.
Note also that the analysis presented in this paper differs significantly (except for the arguments via inverse monotone matrices) from the one for positivity preserving (flux corrected) finite volume schemes which are discussed broadly in the literature, see e.g.\ \cite{LipnikovSSV07,NakshatralaV08,SchneiderGHTT18,ShengG22}.

The paper is organized as follows: Section \ref{sec:hdg_discr} introduces the diffusion equation and the HDG
discretization on hypergraphs. It also states the main theorem of this work.
Section \ref{SEC:hdg_is_fd} proves the theorem for graphs by establishing an
equivalence relation to finite differences. Section \ref{SEC:pos_pres_hg}
extends the proof to multi-dimensional cases, while Section \ref{SEC:counter} presents counterexamples for nonpositive discretizations. Section \ref{SEC:numerics} provides numerical results that confirm the theoretical findings. Finally, the paper concludes with some final remarks.

% ---------------------------------------------------------------------------------------------------
\section{Problem formulation and HDG discretization}\label{sec:hdg_discr}
% ---------------------------------------------------------------------------------------------------
% 
In this paper, the notation $\Omega$ will denote either
\begin{enumerate}[leftmargin=*]
 \item\label{IT:graph} a graph of straight, one-dimensional ($d=1$) edges
$\edge \in \setEdge$ and zero-dimensional nodes $\node \in \setNode$, or
 \item\label{IT:hypergraph} a hypergraph in the sense of \cite{RuppGK22}, where all hyperedges
$\edge \in \setEdge$ are $d$-dimensional ($d>1$) polyhedra and the hypernodes $\node \in \setNode$ coincide with the $(d-1)$-dimensional faces of the
hyperedges, or
\item\label{IT:std_domain} a domain in $\IR^d$ comprising a mesh of polyhedral
elements $\edge \in \setEdge$ with associated $(d-1)$-dimensional faces
$\node \in \setNode$.
\end{enumerate}
The local dimension $d$ of $\Omega$ may or may not coincide with the dimension of the space $\IR^D$ in which $\Omega$ is embedded, i.e., $d \le D$. In all three cases, $\Omega$ is assumed to be connected and the boundary of any $\edge\in\setEdge$ is the union of some elements of $\setNode$. The $d$-dimensional manifolds $\edge\in\setEdge$ are assumed to be open, i.e.,
$\edge\cap\partial\edge=\emptyset$ for any $\edge\in\setEdge$. The hypernodes 
$\node\in\setNode$ are assumed to be closed. Those $\node\in\setNode$ which are
contained in the boundary of only one element of $\setEdge$ form the boundary
$\partial\Omega$ of $\Omega$. The boundary $\partial\Omega$ is decomposed into
a Dirichlet part $\Gamma_\textup D$ and a Neumann part $\Gamma_\textup N$ and
it is assumed that $\Gamma_\textup D$ is the union of a positive number of
elements of $\setNode$. The Neumann part $\Gamma_\textup N$ may be empty. It is
assumed that if $N\cap\partial E$ has a positive $(d-1)$-dimensional measure
for some $\node\in\setNode$ and $\edge\in\setEdge$, then $N\subset\partial E$.
Note that, in the cases (1) and (2), a (hyper)node $\node\in\setNode$ may be
contained in the boundaries of more than two (hyper)edges $\edge\in\setEdge$.
The parameter $h =\max_{\edge \in \setEdge} \operatorname{diam} (E)$ refers to
the maximum diameter of $\edge$. This notation is useful as one can refine the
(hyper)edges of a (hyper)graph similarly to elements of triangulations of
$d$-dimensional domains, cf.\ \cite{RuppGK22}.

We consider the stationary diffusion equation
\begin{subequations}\label{EQ:diffusion}
\begin{align}
 - \div( \kappa \nabla u ) & = f && \text{ in all } \edge\in\setEdge,\\
   \kappa \nabla u \cdot \Normal & = 0 && \text{ on } \Gamma_\textup N,\\
 u & = g_\textup D && \text{ on } \Gamma_\textup D, \\
 u|_{E_1}^{} & =u|_{E_2}^{}  && \text{ on all } 
 \node\in\setNode,\ N\subset\partial E_1\cap\partial E_2,\label{eq:cont} \\
 \sum_{\edge\in\setEdge,\,\partial E\supset N}\,
 (\kappa \nabla u|_E^{}\cdot\Normal_E)|_N^{} & =0  && \text{ for all }
 \node\in\setNode,\ N\not\subset\partial\Omega,\label{eq:conserv}
\end{align}
\end{subequations}
where $0 < \kappa_0 \le \kappa \in L^\infty(\Omega)$ with $\kappa_0 \in \IR$,
$0\le f\in L^2(\Omega)$, and $0 \le g_\textup D \in H^{1/2}(\Gamma_\textup D)$
are nonnegative functions. Here, 
$g_\textup D \in H^{1/2}(\Gamma_\textup D)$ is used as a shortcut for the
fact that, for any $N\subset\partial\Omega\cap\partial\edge$, the restriction 
$g_\textup D|_N^{}$ is the trace on $N$ of a function from $H^1(E)$. Finally,
$\Normal$ denotes an outward-pointing unit normal vector to
$\partial\Omega$ and $\Normal_E$ is the outward-pointing unit normal vector to
$\partial E$. The conditions \eqref{eq:cont} and \eqref{eq:conserv} imply that,
at each junction, the solution is continuous and conservative, i.e., all fluxes
sum to zero.

To formulate problem \eqref{EQ:diffusion} weakly, a natural space for $u$ 
is an analogue of the Sobolev space $H^1$ defined on the graph $\Omega$ by
\begin{equation*}
 \spaceH = \left\{ v \in \bigoplus_{\Edge \in \setEdge} H^1(\Edge) \colon 
 \gamma_1(v|_{E_1}^{}) = \gamma_2(v|_{E_2}^{}) \ \text{ on all }
 \node\in\setNode,\ N\subset\partial E_1\cap\partial E_2\right\},
\end{equation*}
where $\gamma_i$ is the trace operator on $H^1(E_i)$, $i=1,2$. We denote
\begin{equation*}
 \spaceH_{g_\textup D} =  \{ v \in\spaceH\colon\trace v|_{\Gamma_\textup D}^{}
 = g_\textup D \}
\end{equation*}
and use the notation $\spaceH_0$ for $\spaceH_{g_\textup D}$ with
$g_\textup D\equiv0$. A straightforward generalization of the standard
derivation of a primal weak formulation reveals that a weak solution to problem
\eqref{EQ:diffusion} is a function $u\in\spaceH_{g_\textup D}$ such that
\begin{equation}\label{eq:weak_form}
 \sum_{\Edge \in \setEdge} \int_\Edge \kappa \nabla u \cdot \nabla v \dbx =
\int_\Omega f v \dbx \qquad \forall v \in \spaceH_0.
\end{equation}
This weak formulation has a unique solution, see \cite{RuppGK22}.

\begin{remark}[Standard domain]
If $\Omega \subset \IR^d$ is a standard domain in the sense of item
(\ref{IT:std_domain}), then $\spaceH = H^1(\Omega)$ and 
$\spaceH_0=\{v\in H^1(\Omega)\colon \trace v|_{\Gamma_\textup D}^{}=0\}$.
Thus, we recover the well-known weak form of the diffusion equation.
\end{remark}

Apart from the unique existence of the weak solution 
$u\in\spaceH_{g_\textup D}$, we also have the following result.
\begin{lemma}\label{LEM:pos_preservation}
The solution
$u \in \spaceH_{g_\textup D}$ of \eqref{eq:weak_form} satisfies $u \ge 0$
almost everywhere.
\end{lemma}
\begin{proof}
Define $u^+(\vec x) = \max\{0, u(\vec x)\}$ and $u^-(\vec x) = \min\{0, u(\vec x)\}$.
Then, it follows from \cite[Lemma 7.6]{GT01} and the assumptions on
$g_\textup D$, that $u^- \in \spaceH_0$ and $\nabla u^+\cdot\nabla u^-=0$
a.e. in $\Edge$ for any $\Edge \in \setEdge$. Consequently,
\begin{equation*}
  0 \le \sum_{\Edge \in \setEdge} \int_\Edge \kappa \nabla u^- \cdot \nabla u^-
\dbx = \sum_{\Edge \in \setEdge} \int_\Edge \kappa \nabla u \cdot \nabla u^-
\dbx = \int_\Omega f u^- \dbx \le 0,
\end{equation*}
which yields the result as it implies that $u^- = 0$ and $u = u^+$.
\end{proof}

We introduce the skeleton
\begin{equation*}
   \Sigma_h=\bigcup_{\node\in\setNode}\,N
   =\bigcup_{\edge\in\setEdge}\,\partial E
\end{equation*}
and the spaces
\begin{equation*}
   \spaceM = \{ \mu\in L^2(\Sigma_h) \colon 
   \mu|_{\partial E}^{}\in H^{1/2}(\partial E)\,\,\forall\edge\in\setEdge\},
   \quad
   \spaceM_0 = \{\mu\in\spaceM\colon\mu=0\,\,\text{on } \Gamma_\textup D\}.
\end{equation*}
Then the weak formulation \eqref{eq:weak_form} can be equivalently written as
the following dual mixed hybrid formulation: Find $(u,\vec q,\lambda)\in
L^2(\Omega)\times\bigoplus\Hdiv(\edge)\times\spaceM$ with $\lambda=g_\textup D$
on $\Gamma_\textup D$ such that
\begin{subequations}\label{eq:mixed_hybrid}
 \begin{align}
  \int_\Edge \kappa^{-1} \vec q \cdot \vec p\dbx - \int_\Edge u \div \vec p\dbx
  & = - \langle\vec p \cdot \Normal_E,\lambda\rangle_{\partial \Edge}
  &&\forall \vec p\in\bigoplus_{\edge\in\setEdge}\Hdiv(\edge),\,\,
  \edge\in\setEdge,\label{eq:mixed_hybrid_primary}\\
  \int_\Edge v\div \vec q \dbx & = \int_\edge f v \dbx 
  &&\forall v\in L^2(\Omega),\,\,\edge\in\setEdge,\label{eq:mixed_hybrid_flux}\\
  \sum_{\edge\in\setEdge}\,
  \langle\vec q \cdot \Normal_E,\mu\rangle_{\partial \Edge} & =0
  &&\forall\mu\in\spaceM_0,
 \end{align}
\end{subequations}
where $\langle\cdot,\cdot\rangle_{\partial \Edge}$ denotes the duality pairing 
between $H^{-1/2}(\partial E)$ and $H^{1/2}(\partial E)$, which is used because
the normal trace of a function $\vec p\in\Hdiv(\edge)$ is an element of
$H^{-1/2}(\partial E)$. Using the results of \cite[Section III.1.1]{BF91}, it 
follows that $(u,\vec q,\lambda)$ solves \eqref{eq:mixed_hybrid} if and only if
$u\in\spaceH_{g_\textup D}$ is a solution of \eqref{eq:weak_form}, $\vec
q|_E^{}=-\kappa\nabla u|_E^{}$ for any $\edge\in\setEdge$, and
$\lambda=u|_{\Sigma_h}^{}$. This shows that \eqref{eq:mixed_hybrid} possesses a
unique solution.

The mixed hybrid formulation \eqref{eq:mixed_hybrid} is the basis for
discretizations of the diffusion equation \eqref{EQ:diffusion} considered in
this paper. To approximate the spaces used in \eqref{eq:mixed_hybrid}, we
introduce finite-dimensional spaces $U(\edge)\subset H^1(\edge)$, $\vec
Q(\edge)\subset H^1(\edge)^d$, and $M(\node)\subset L^2(\node)$ for each
$\edge\in\setEdge$ or $\node\in\setNode$ and set
\begin{align*}
 U_h & = \{ v_h \in L^2(\Omega)\colon v_h|_\edge^{}\in U(\edge)\,\,\forall \edge \in \setEdge \},\\
 \vec Q_h & = \{ \vec p_h \in L^2(\Omega)^d\colon\vec p_h|_\edge^{}\in
\vec Q(\edge)\,\,\forall \edge \in \setEdge \},\\
 M_h & = \{ \mu_h \in L^2(\Sigma_h)\colon\mu_h|_\node^{}\in M(\node)\,\,\forall
\node \in \setNode\}.
\end{align*}
Moreover, we define the space
\begin{equation*}
   M_h^0 = \{\mu_h\in M_h\colon\mu_h=0\,\,\text{on } \Gamma_\textup D\}
\end{equation*}
and approximate $g_\textup D$ by a function $g_{{\textup D},h}\ge0$ defined on
$\Gamma_\textup D$ such that $g_{{\textup D},h}|_\node^{}\in M(\node)$ for any
$\node\subset\Gamma_\textup D$. Then the HDG discretization of
\eqref{EQ:diffusion} reads: Find
$(u_h,\vec q_h,\lambda_h)\in U_h\times\vec Q_h\times M_h$ with
$\lambda_h=g_{{\textup D},h}$ on $\Gamma_\textup D$ such that
\begin{subequations}\label{EQ:hdg_scheme}
 \begin{align}
  &\int_\Edge\kappa^{-1}\vec q_h\cdot\vec p_h\dbx-\int_\Edge u_h\div\vec p_h\dbx
   = - \int_{\partial\Edge}\lambda_h\,\vec p_h\cdot\Normal_E\ds\quad
  \forall \vec p_h\in\vec Q_h,\,\,\edge\in\setEdge,\label{EQ:hdg_primary}\\
  &\int_\Edge v_h\div \vec q_h \dbx + \tau_E \int_{\partial \Edge} u_h v_h \ds
   = \tau_E \int_{\partial \Edge} \lambda_h v_h \ds + \int_\edge f v_h \dbx 
  \quad\forall v_h\in U_h,\,\,\edge\in\setEdge,\label{EQ:hdg_flux}\\
  &\sum_{\edge\in\setEdge}\,\int_{\partial\Edge}\mu_h\,(\vec q_h\cdot\Normal_E 
  + \tau_E(u_h-\lambda_h))\ds =0 \quad\forall\mu_h\in M_h^0.
  \label{EQ:hdg_global}
 \end{align}
\end{subequations}
The discretization comprises an HDG stabilization with penalty coefficients
$\tau_E\ge0$. If not stated otherwise, it will be always assumed that
$\tau_E>0$. It is important to include this stabilization not only in the
flux balance \eqref{EQ:hdg_flux} but also in the global mass-conservation
constraint \eqref{EQ:hdg_global} which couples the local equations
\eqref{EQ:hdg_primary} and \eqref{EQ:hdg_flux}. Note that \eqref{EQ:hdg_global}
can be equivalently written in the form
\begin{equation}
  \sum_{\edge\in\setEdge,\,\partial E\supset N}\,
  \int_\node\mu\,(\vec q_h|_E^{}\cdot\Normal_E 
  + \tau_E(u_h|_E^{}-\lambda_h))\ds =0\quad\forall\mu\in M(\node),
  \label{EQ:hdg_global_equiv}
\end{equation}
for all $\node\in\setNode$ such that $\node\not\subset\Gamma_\textup D$.

\begin{theorem}[Unique solvability of the HDG discretization]\label{thm:hdg_solvability}
Let $\tau_E>0$ and $\nabla U(\edge)\subset\vec Q(\edge)$ for all
$\edge\in\setEdge$. Then the HDG discretization \eqref{EQ:hdg_scheme} possesses
a unique solution.
\end{theorem}

\begin{proof}
The discretization \eqref{EQ:hdg_scheme} can be written as a linear system with
a square matrix for the degrees of freedom of the solution 
$(u_h,\vec q_h,\lambda_h)$. Therefore, it suffices to show that the homogeneous
problem with $f=0$ and $g_{{\textup D},h}=0$ has only a trivial solution.
First, we set $\vec p_h=\vec q_h$, $v_h=u_h$, and $\mu_h=-\lambda_h$, sum both 
\eqref{EQ:hdg_primary} and \eqref{EQ:hdg_flux} over all $\edge\in\setEdge$ and
add all three relations together. This yields
\begin{equation*}
   \sum_{\edge\in\setEdge}\,\int_\Edge \kappa^{-1}|\vec q_h|^2\dbx
   +\sum_{\edge\in\setEdge}\,\tau_E\int_{\partial\Edge}(u_h-\lambda_h)^2\ds =0.
\end{equation*}
Consequently $\vec q_h=\vec 0$, $u_h\in\spaceH_0$, and
$\lambda_h=u_h|_{\Sigma_h}^{}$. Then, integrating by parts in
\eqref{EQ:hdg_primary}, gives
\begin{equation*}
  \int_\Edge \vec p_h\cdot\nabla u_h\dbx=0\quad
  \forall \vec p_h\in\vec Q_h,\,\,\edge\in\setEdge.
\end{equation*}
Thus, for any $\edge\in\setEdge$, setting $\vec p_h|_E^{}=\nabla u_h|_E^{}$
leads to $\nabla u_h|_E^{}=0$ and hence $u_h|_E^{}$ is constant. Since
$u_h\in\spaceH_0$, it follows that $u_h=0$, which finishes the proof.
\end{proof}

The assumptions of Theorem~\ref{thm:hdg_solvability} are satisfied for most
cases considered in this paper. There are also many choices of the spaces 
$U(\edge)$, $\vec Q(\edge)$, and $M(\node)$ for which the HDG method
\eqref{EQ:hdg_scheme} is well-posed with $\tau_E=0$, see \cite{CockburnGL09}.
From these possibilities only the choice $U(\edge)=\polynomialsp_0(\edge)$,
$\vec Q(\edge)=\polynomialsRT_0(\edge)$, and $M(\node)=\polynomialsp_0(\node)$
will be briefly considered where $\polynomialsRT_0(\edge)$ is the lowest-order
Raviart-Thomas space defined by
\begin{alignat*}{2}
   \polynomialsRT_0(\edge) &= \{\vec q+a\vec x\colon
   \vec q\in\polynomialsp_0(\edge)^d,\,a\in\polynomialsp_0(\edge),\,
   \vec x\in\edge\}\qquad&&\mbox{for simplices}\,,\\
   \polynomialsRT_0(\edge) &= \{\vec q+(a_1 x_1,\dots,a_d x_d)^T\colon
   \vec q,\vec a\in\polynomialsp_0(\edge)^d,\,
   \vec x\in\edge\}\qquad&&\mbox{for $n$-rectangles}\,.
\end{alignat*}

It follows from Lemma~\ref{LEM:pos_preservation} that the solution of the
mixed hybrid formulation \eqref{eq:mixed_hybrid} satisfies $u\ge0$ and
$\lambda\ge0$ almost everywhere (with respect to the $d$-dimensional and
$(d-1)$-dimensional measures, respectively). Our target is to investigate the
validity of these nonnegativity conditions for the above HDG discretization
with various spaces and we will establish both positive and negative results.
We conclude this section by stating the main theorem of this contribution. The
proof will be given in the next sections.

\begin{theorem}[Nonnegativity of solutions]\label{TH:main_result}\
 \begin{enumerate}[leftmargin=*]
\item\label{IT:graph_first_order} 
Let $\Omega$ be a graph in the sense of item (\ref{IT:graph}) from the
beginning of Section~\ref{sec:hdg_discr}, $U(\edge)=\polynomialsp_1(\edge)$,
$Q(\edge)=\polynomialsp_1(\edge)$, and $\tau_E>0$ for all $\edge\in\setEdge$, 
$M(\node)=\IR$
for all $\node\in\setNode$, and assume that $\kappa$ is constant on each
$\edge\in\setEdge$ (but not necessarily globally constant). Let $u_h$ and
$\lambda_h$ be the solutions of the HDG method \eqref{EQ:hdg_scheme} and denote
by $u_{0,\edge}$ the mean value of $u_h$ on any $\edge\in\setEdge$. Then 
$\lambda_h \ge 0$ and $u_{0,\edge} \ge 0$ for all $\edge\in\setEdge$. If, 
additionally, $f$ is constant on each $\edge \in \setEdge$, then $u_h \ge 0$.
\item\label{IT:graph_zero_order} 
Let $\Omega$, $\tau_E$, $\kappa$, and $M(\node)$ be like in item
(\ref{IT:graph_first_order}). For all $\edge\in\setEdge$, let either
 \begin{itemize}
  \item $U(\edge)=\polynomialsp_1(\edge)$ and $Q(\edge)=\polynomialsp_0(\edge)$, or
  \item $U(\edge)=\polynomialsp_0(\edge)$ and $Q(\edge)=\polynomialsp_1(\edge)$, or
  \item $U(\edge)=\polynomialsp_0(\edge)$ and $Q(\edge)=\polynomialsp_0(\edge)$.
 \end{itemize}
Then
the solutions $u_h$ and $\lambda_h$ of the HDG method \eqref{EQ:hdg_scheme}
satisfy $u_h\ge0$ and $\lambda_h\ge0$.
\item\label{IT:zero_order}Let $\Omega$ be a hypergraph or a domain in
$\IR^d$ in the sense of items (\ref{IT:hypergraph}) and (\ref{IT:std_domain})
from the beginning of Section~\ref{sec:hdg_discr}, and let
$U(\edge)=\polynomialsp_0(\edge)$, $\vec Q(\edge)=\polynomialsp_0(\edge)^d$, $\tau_\edge > 0$ for all $\edge\in\setEdge$, and $M(\node)=P_0(\node)$, for all $\node\in\setNode$ (LDG-H).
Let one of the following conditions hold for each $\edge \in \setEdge$:
  \begin{enumerate}
   \item\label{IT:cond_angle} We have
$\Normal_E|_\node^{}\cdot\Normal_E|_{\node'}^{}\le0$ for all 
$\node,\node'\in\setNode$ with $\node\neq\node'$ and
$\node,\node'\subset\partial\edge$.
\item\label{IT:tau_cond} The stabilization parameters $\tau_E$ are chosen such that
 \begin{equation}\label{EQ:tau_cond}
   \tau_\edge\ge\frac{|\partial\edge|}{\int_\edge \kappa^{-1} \dbx}.
  \end{equation}
 \end{enumerate}
 Then the solutions $u_h$ and $\lambda_h$ of the HDG method
\eqref{EQ:hdg_scheme} satisfy $\lambda_h \ge 0$ and $u_h \ge 0$.
\item\label{IT:RT}Let $\Omega$, $U(\edge)$, and $M(\node)$ be like in
item~(\ref{IT:zero_order}). Let all $\edge\in\setEdge$ be simplices and let
$\vec Q(\edge)=\polynomialsRT_0(\edge)$, $\tau_\edge=0$, and 
$\kappa|_\edge^{}$ be constant for all $\edge\in\setEdge$ (RT-H). If the
condition (\ref{IT:cond_angle}) from the previous item holds, then the
solutions $u_h$ and $\lambda_h$ of the HDG method \eqref{EQ:hdg_scheme} satisfy
$\lambda_h \ge 0$ and $u_h \ge 0$.
\item\label{IT:q0_p0_p0}Let $\Omega$ be a hypergraph or a domain in
$\IR^2$ in the sense of items (\ref{IT:hypergraph}) and (\ref{IT:std_domain})
from the beginning of Section~\ref{sec:hdg_discr}, and let all hyperedges be
rectangles. Let $\kappa_\edge:=\kappa|_\edge^{}$ be constant for each
$\edge\in\setEdge$. Let $U(\edge)=\polynomialsp_0(\edge)$ for all
$\edge\in\setEdge$, $M(\node)=\polynomialsp_0(\node)$ for all
$\node\in\setNode$, and let, for all $\edge\in\setEdge$, either 
$\vec Q(\edge)=\polynomials_1(\edge)^2$, or 
$\vec Q(\edge)=\polynomialsp_1(\edge)^2$ or 
$\vec Q(\edge)=\polynomialsRT_0(\edge)$ (rectangular version). Then all these
three HDG methods are equivalent and there exists a positive constant
$C(\varrho)$ depending only on the aspect ration $\varrho$ of the rectangles
$\edge$ such that, if
\begin{equation}\label{eq:q0_p0_p0_cond}
   \tau_E\ge C(\varrho)\frac{\kappa_\edge}{\operatorname{diam}(\edge)}\qquad
   \forall\edge\in\setEdge,
\end{equation}
then the solutions $u_h$ and $\lambda_h$ of the HDG method
\eqref{EQ:hdg_scheme} satisfy $\lambda_h \ge 0$ and $u_h \ge 0$. If $\edge$ is
a square of side length $h$, then the condition \eqref{eq:q0_p0_p0_cond} can be
replaced by
\begin{equation}\label{eq:q0_p0_p0_cond_square}
   \tau_E\ge\frac{2\kappa_\edge}h.
\end{equation}
\end{enumerate}
\end{theorem}

Loosely speaking, the content of Theorem~\ref{TH:main_result} contains several main components if \( \vec q \) is approximated by element-wise constant, linears, or Raviart-Thomas functions:
\begin{itemize}
 \item If the HDG method is used to discretize a diffusion equation with edge-wise constant diffusion coefficient on a graph, and \( u \) is approximated by piecewise constants, then \( \lambda_h \ge 0 \) and \( u_h \ge 0 \). Moreover, if \( u \) is approximated by piecewise linear functions and \( f \) is constant on each edge, then \( u_h \ge 0 \) as well.
 \item If the HDG method is used to discretize a diffusion equation with edge-wise constant diffusion coefficient on a hypergraph or a domain in \( \IR^d \), all angles are \( \le \tfrac\pi2 \), and \( u_h \) is approximated by element-wise constants, then \( u_h \ge 0 \) and \( \lambda_h \ge 0 \).
 \item If the angle condition in the aforementioned item is violated, we can construct a positivity preserving HDG method by using piecewise constant functions for \( u_h \) and a suitable stabilization parameter \( \tau_E \) that satisfies the condition \eqref{EQ:tau_cond}.
\end{itemize}
If the above conditions are not satisfies, then positivity cannot be guaranteed in general as demonstrated in Section~\ref{SEC:counter}.

\begin{remark}
In practical applications, one can fulfill the assumptions on piecewise
constant data by approximating $\kappa$ or $f$ by piecewise constant functions
defined by local $L^2$ projections to constants or by replacing the functions
on any $\edge$ by their values at the barycenter of $\edge$. The latter
possibility corresponds to using the midpoint rule for quadrature.

Condition (\ref{IT:cond_angle}) is satisfied for tensorial meshes and for
simplicial meshes without obtuse angles. Under this condition, the RT-H method 
on simplicial meshes is positivity preserving also for $\tau_\edge \equiv 0$.
On simplicial meshes with obtuse angles, the RT-H method has to be used with
appropriate $\tau_\edge>0$ to guarantee that the solution is nonnegative.

Concerning item (\ref{IT:zero_order}), typical choices for $\tau_E$ in HDG
methods are $\tau_E = \mathcal O(1)$ and $\tau_E = \mathcal O(h^{-1})$,
respectively. The latter choice is motivated by the fact that the stabilization
term in \eqref{EQ:hdg_flux} is a lumped approximation to the bulk component of
the HDG stabilization. If condition (\ref{IT:cond_angle}) is satisfied, then 
the former choice should be considered, whereas the latter one corresponds to
condition (\ref{IT:tau_cond}).  At least on regular, and geometrically
conforming meshes, and for equal-order spaces of polynomial degree $k$, we can
immediately deduce from \cite[Th.\ 2.1, Lem.\ 3.2, Th.\ 4.1]{CockburnGS10}
that
\begin{align*}
 \| \vec q_h - \vec q \|_{0,E}^{} & \lesssim h_E^{k + 1} | \vec q |_{k+1,E}^{} + \tau_E
h_E^{k + 1} | u |_{k + 1,E}^{}, \\
 \| u_h - u \|_{0,E}^{} & \lesssim h_E^{k + 1} | u |_{k + 1,E}^{} + \frac{h_E^{k +
1}}{\tau_E} | \nabla \cdot \vec q |_{k,E}^{} + h^{\min(k,1)} \| \vec q_h - \vec
q \|_{0,E}^{},
\end{align*}
yielding first order convergence for $u_h$ and $\vec q_h$ if the polynomial approximation spaces are chosen to be of degree $k=0$ and $\tau = \mathcal O(1)$. Additionally, \cite{CockbDG08} state the following table for rates of convergence:
\begin{table}[h!]
 \begin{tabular}{cccrl}
  \toprule
   method & \( \| \vec q - \vec q_h \|_{L^2(\Omega)} \) & \( \| u_h - u \|_{L^2(\Omega)} \) & \multicolumn{2}{c}{condition} \\
   \midrule
   RT-H$_k$ & \( k+1 \) & \( k+1 \) & \(k \ge 0\), &\( \tau_\edge = 0 \) \\
   LDG-H$_k$ & \( k \) & \( k+1 \) &  \( k \ge 1\), & \( \tau_\edge = \mathcal O(\tfrac1h) \) \\
   LDG-H$_k$ & \( k + \tfrac12 \) & \( k+1 \) &  \( k \ge 0 \), &\(\tau_\edge = \mathcal O(1) \) \\
   BDM-H$_k$ & \( k+1 \) & \( k \) &  \( k \ge 1\), &\(\tau_\edge = 0 \) \\
   \bottomrule
 \end{tabular}
\end{table}

Notably, conditions \eqref{EQ:tau_cond}, \eqref{eq:q0_p0_p0_cond}, and
\eqref{eq:q0_p0_p0_cond_square} can influence the convergence rates.
\end{remark}

\begin{remark}\label{rem:coefficients}
A large part of the analysis in this paper will be performed for
two-dimensional rectangular hyperedges $\edge\in\setEdge$ and spaces
$U(\edge)\subset\polynomials_1(\edge)$, 
$\vec Q(\edge)\subset\polynomials_1(\edge)^2$, and 
$M(\node)\subset\polynomialsp_1(\node)$. In this case, assuming that $\kappa$
takes a constant value $\kappa_\edge>0$ on each $\edge\in\setEdge$, it is not
difficult to derive explicit formulas for $u_h$ and $\vec q_h$ solving
\eqref{EQ:hdg_primary} and \eqref{EQ:hdg_flux} for a given function
$\lambda_h\in M_h$. To this end, we consider a local coordinate system with axes
parallel to the edges of the rectangle $\edge$ and origin at the barycenter of
$\edge$. Thus, with respect to this coordinate system, we have
$\edge=(-h_1/2,h_1/2)\times(-h_2/2,h_2/2)$ for some $h_1,h_2>0$. To simplify
the derivation, the coordinate axes will be denoted $x_1$ and $x_2$ although at
other places in this paper the notation $x$ and $y$ will be preferred. The
notations $\vec e_1$ and $\vec e_2$ will be used for unit vectors in the
directions of the coordinate axes $x_1$ and $x_2$, respectively.  The edges of
$\edge$ are denoted by $\node_1^\pm$ and $\node_2^\pm$ with the convention that
$\Normal_E|_{\node_i^\pm}^{}=\pm\vec e_i$, $i=1,2$, see
Fig.~\ref{fig:local_notation}.
\begin{figure}[t]
\centering
\begin{tikzpicture}
   \draw[line width=0.3mm]
      (0,0) -- (4,0) -- (4,3) -- (0,3) -- cycle;
   \draw [-{Latex[length=3mm]}, line width=0.2mm] (-1.5,1.5) -- (5.5,1.5);
   \draw [-{Latex[length=3mm]}, line width=0.2mm] (2,-1.5) -- (2,4.5);
   \node {} (0.05,2.9) node[below right] {$\edge$};
   \node {} (5.2,1.4) node[below] {$x_1$};
   \node {} (1.95,4.2) node[left] {$x_2$};
   \node {} (4,2.4) node[right] {$\node_1^+$};
   \node {} (0,2.4) node[left] {$\node_1^-$};
   \node {} (3.4,3) node[above] {$\node_2^+$};
   \node {} (3.4,0) node[below] {$\node_2^-$};
   \node {} (4,1.5) node[below right] {$\frac{h_1}2$};
   \node {} (0,1.5) node[below left] {$-\frac{h_1}2$};
   \node {} (2,3) node[above left] {$\frac{h_2}2$};
   \node {} (2,0) node[below left] {$-\frac{h_2}2$};
\end{tikzpicture}
\caption{Local coordinate system for a rectangular hyperedge $\edge$ and
notation for hypernodes.}
\label{fig:local_notation}
\end{figure}
On $\edge$, the functions $u_h$ and $\vec q_h$ have the form
\begin{subequations}\label{eq:uh_qh_general}
\begin{align}
 u_h|_\edge^{}(x_1,x_2) & = u_0 + a_1 x_1 + a_2 x_2 + b x_1 x_2, \\
 \vec q_h|_\edge^{}(x_1,x_2) & = \vec q_0 + 
 \begin{pmatrix}
   a^1_1 x_1 + a^1_2 x_2 + b^1 x_1 x_2\\[1mm]
   a^2_1 x_1 + a^2_2 x_2 + b^2 x_1 x_2
 \end{pmatrix}.
\end{align}
\end{subequations}
Note that the considered basis functions $1$, $x_1$, $x_2$, $x_1 x_2$ of
$Q_1(\edge)$ are orthogonal in both $L^2(\edge)$ and $L^2(\partial\edge)$, i.e.,
\begin{align*}
 &\int_\edge x_1\dbx=\int_\edge x_2\dbx=\int_\edge x_1x_2\dbx
 =\int_\edge x_1^2x_2\dbx=\int_\edge x_1x_2^2\dbx=0,\\
 &\int_{\partial\edge}x_1\ds=\int_{\partial\edge}x_2\ds
 =\int_{\partial\edge}x_1x_2\ds=\int_{\partial\edge}x_1^2x_2\ds
 =\int_{\partial\edge}x_1x_2^2\ds=0.
\end{align*}
We observe that testing \eqref{EQ:hdg_primary} with
\begin{itemize}
 \item $\vec p_h = \vec e_k$, $k=1,2$, yields
 \begin{equation}\label{eq:q0}
  \vec q_0 = - \frac{\kappa_\edge}{|\edge|} 
  \int_{\partial \edge} \lambda_h \Normal_\edge \ds,
 \end{equation}
 \item $\vec p_h = x_i \vec e_k$, $i,k\in\{1,2\}$, yields
 \begin{equation*}\label{eq:aik}
  a^k_i  = \delta_{ik} \frac{12\kappa_\edge}{h_i^2} u_0 
  - \frac{12\kappa_\edge}{|\edge|\,h_i^2}\left(
  \int_{\node_k^+}\lambda_h x_i\ds-\int_{\node_k^-}\lambda_h x_i\ds\right),
 \end{equation*}
 \item $\vec p_h = x_1 x_2 \vec e_k$, $k\in\{1,2\}$, yields
 \begin{equation*}\label{eq:bk}
  b^k = a_i\frac{12\kappa_\edge}{h_k^2} 
  -\frac{72\kappa_\edge}{h_i^3h_k^2}
  \int_{\node_k^+\cup\node_k^-}\lambda_h x_i\ds,
 \end{equation*}
 where $i\neq k$, $i\in\{1,2\}$.
\end{itemize} 
On the other hand, testing \eqref{EQ:hdg_flux} with
\begin{itemize}
 \item $v_h=1$, yields
\begin{equation}\label{eq:u0}
 (a^1_1+a^2_2)|\edge| + \tau_\edge|\partial\edge|u_0  
 = {\tau_\edge}\int_{\partial\edge} \lambda_h\ds + \int_\edge f \dbx,
\end{equation}
 \item $v_h = x_i$, $i\in\{1,2\}$, yields
\begin{align*}
 b^k\frac{h_i^3h_k}{12} + a_i{\tau_\edge} \frac{h_i^3+3h_i^2h_k}6
 = {\tau_\edge} \int_{\partial \edge} \lambda_h x_i \ds + \int_\edge f x_i\dbx,
\end{align*}
where $k\neq i$, $k\in\{1,2\}$,
 \item $v_h = x_1 x_2$, yields
\begin{equation*}
 b = \frac{24}{h_1^2h_2^2(h_1+h_2)}\left(\int_{\partial\edge}\lambda_h x_1 x_2 \ds
 + \frac1{\tau_\edge} \int_{\edge} f x_1 x_2 \dbx\right).
\end{equation*}
\end{itemize}
Of course, if $U(\edge)\neq\polynomials_1(\edge)$ or 
$\vec Q(\edge)\neq\polynomials_1(\edge)^2$, then only those of the above
formulas are valid for which the used test functions belong to the respective
space. Combining the above relations, it is easy to compute the coefficients in
\eqref{eq:uh_qh_general}.
\end{remark}

% ---------------------------------------------------------------------------------------------------
\section{HDG is (mimetic) finite difference method on graphs}\label{SEC:hdg_is_fd}
% ---------------------------------------------------------------------------------------------------
% 
To prove Theorem \ref{TH:main_result} on graphs, we prove that our linear HDG method coincides with the finite difference approach. That is, we have:
\begin{theorem}\label{TH:hdg_is_fd}
Let $\Omega$ be a graph in the sense of item (\ref{IT:graph}) from the
beginning of Section~\ref{sec:hdg_discr}, $U(\edge)=\polynomialsp_1(\edge)$ and
$Q(\edge)=\polynomialsp_0(\edge)$ or $Q(\edge)=\polynomialsp_1(\edge)$ for all 
$\edge\in\setEdge$, $M(\node)=\IR$
for all $\node\in\setNode$, and let $\tau_E>0$ and $\kappa$ be edgewise constant, i.e.,
$\kappa|_\edge^{} = \kappa_\edge \in \IR$ for all $\edge\in\setEdge$. For
$\lambda_h\in M_h$ and $\node\in\setNode$, denote
$\lambda_\node=\lambda_h(\node)$. For $\node\in\setNode$ and an edge
$\edge\in\setEdge$ with $\partial\edge\ni\node$, denote by 
$\node'\in\setNode\cap\partial\edge$ with $\node' \neq \node$ the other
endpoint of $E$. Furthermore, let $\varphi^\edge_\Node$ be the linear function
on $\edge$ that satisfies $\varphi^\edge_\Node(\Node) = 1$ and
$\varphi^\edge_\Node(\node') = 0$. Then $\lambda_h\in M_h$ is a solution of the
HDG method \eqref{EQ:hdg_scheme} if and only if $\lambda_\node = g_\textup
D(\node)$ for $\node\in\setNode\cap\Gamma_\textup D$ and $\lambda_h$ satisfies
the finite difference equations
\begin{equation}\label{EQ:fd_system}
  \sum_{\edge\in\setEdge,\,\partial E\ni\node}\,
  \frac{\kappa_\edge}{|\edge|} (\lambda_\Node - \lambda_{\node'}) 
  = \sum_{\edge\in\setEdge,\,\partial E\ni\node}\, 
  \int_\edge f \varphi^\edge_\Node\dx
\end{equation}
for every $\node \in \setNode$ with $\node \not \in \Gamma_\textup D$.
\end{theorem}
\begin{proof}
Let us consider \eqref{EQ:hdg_primary}, \eqref{EQ:hdg_flux} for an edge 
$\edge\in\setEdge$, that we isometrically map to the interval $(a,b)$ with
$b - a = |\edge|$. Denote the $\polynomialsp_1(\edge)$ basis function
associated with the node $a$ by $\varphi_a$ and the $\polynomialsp_1(\edge)$
basis function associated with the node $b$ by $\varphi_b$. Furthermore,
observe that $\vec p_h$, $\vec q_h$ are scalar functions which will be
denoted by $p_h$, $q_h$, respectively. Throughout the proof, the functions
$u_h$, $q_h$, and $\lambda_h$ are tacitly assumed to be restricted to $E$ if
they are evaluated at $a$ or $b$.

Let us start with the case $Q(\edge)=P_1(\edge)$.
Testing \eqref{EQ:hdg_flux} with $v_h = 1$ yields
\begin{equation*}
  \frac{2 {\tau_E}}{|\edge|} \int_\edge u_h \dx = {\tau_E} (u_h(a) + u_h(b)) =
{\tau_E} (\lambda_h(a) + \lambda_h(b)) +\int_\edge f \dx - \int_\edge q_h^\prime \dx.
\end{equation*}
Since $p_h^\prime|_E^{}$ is constant, we can plug this relation into
\eqref{EQ:hdg_primary} to obtain
 \begin{multline*}
  \int_\edge \kappa_E^{-1} q_h p_h\dx - p_h^\prime|_E^{} \frac{|\edge|}{2 {\tau_E}} \left[
{\tau_E} (\lambda_h(a) + \lambda_h(b)) + \int_\edge f \dx - \int_\edge q_h^\prime \dx
\right] \\
  = \lambda_h(a) p_h(a) - \lambda_h(b) p_h(b),
\end{multline*}
which we can rewrite as
\begin{multline}\label{eq:rel1}
  \int_\edge \kappa_E^{-1} q_h p_h\dx 
  + \frac{|\edge|}{2 {\tau_E}} \int_\edge p_h^\prime q_h^\prime \dx \\
  = \frac{|\edge|}{2 {\tau_E}} 
  \left[ {\tau_E} (\lambda_h(a) + \lambda_h(b)) p_h^\prime|_E^{}
  + \int_\edge f p_h^\prime \dx \right] + \lambda_h(a) p_h(a) - \lambda_h(b) p_h(b).
\end{multline}
Since we consider linear polynomials, we can write 
$q_h|_E^{} = \alpha + \beta r$ for some $\alpha, \beta \in \IR$ and 
$r = \varphi_b - \varphi_a$. Testing with $p_h = 1$ and using the fact that
$\int_\edge r \dx = 0$ allows us to deduce that
\begin{equation*}
  \int_\edge \kappa_E^{-1} (\alpha + \beta r) \dx 
  = \frac{|\edge|}{\kappa_E} \alpha = \lambda_h(a) - \lambda_h(b), 
\end{equation*}
which gives
\begin{equation}\label{eq:rel2}
   \alpha = \frac{{\kappa_E}}{|\edge|} (\lambda_h(a) - \lambda_h(b)).
\end{equation}
Observing that $r^\prime = 2 / |\edge|$ and $\int_\edge r^2\dx = |\edge| / 3$, it
follows from \eqref{eq:rel1} with $p_h|_E^{}=r$ that
 \begin{equation*}
  \beta \left[ \int_\edge \kappa_E^{-1} r^2 \dx + \frac{|\edge|}{2 {\tau_E}}
\int_\edge (r^\prime)^2 \dx \right] = \frac1{\tau_E} \int_\edge f \dx,
 \end{equation*}
which implies that
\begin{equation}\label{eq:rel4}
  \beta = \frac{3 {\kappa_E}}{6 {\kappa_E} + {\tau_E} |\edge|} \int_\edge f \dx.
\end{equation}
Having computed $q_h|_E^{}$, we can obtain $u_h(a)$ and $u_h(b)$ by testing
\eqref{EQ:hdg_flux} with $v_h|_E^{}=\varphi_a$ and $v_h|_E^{}=\varphi_b$
yielding
\begin{subequations}\label{EQ:bdr_vals}
 \begin{align}
  {\tau_E} u_h(a) & = {\tau_E} \lambda_h(a) + \int_\edge f \varphi_a \dx -
\int_\edge q_h^\prime \varphi_a \dx \notag\\
  & = {\tau_E} \lambda_h(a) + \int_\edge f \varphi_a \dx - \frac{3 {\kappa_E}}{6 {\kappa_E} + {\tau_E} |\edge|} \int_\edge f \dx,\\
  {\tau_E} u_h(b) & = {\tau_E} \lambda_h(b) + \int_\edge f \varphi_b \dx - \frac{3 {\kappa_E}}{6 {\kappa_E} + {\tau_E} |\edge|} \int_\edge f \dx.
 \end{align}
\end{subequations}
Since $(\vec q_h\cdot\Normal_E)(a)=-q_h(a)$ and 
$(\vec q_h\cdot\Normal_E)(b)=q_h(b)$, the relations \eqref{EQ:bdr_vals} and the
expression for $q_h|_E^{}$ imply that
\begin{subequations}\label{eq:fluxes}
 \begin{align}
  (\vec q_h\cdot\Normal_E + {\tau_E} (u_h- \lambda_h))(a) & = \frac{{\kappa_E}}{|\edge|} (\lambda_h(b) - \lambda_h(a)) + \int_\edge f \varphi_a \dx, \\
  (\vec q_h\cdot\Normal_E + {\tau_E} (u_h - \lambda_h))(b)
  & = \frac{{\kappa_E}}{|\edge|} (\lambda_h(a) - \lambda_h(b)) + \int_\edge f \varphi_b \dx.
 \end{align}
\end{subequations}

If $Q(\edge)=P_0(\edge)$, then \eqref{eq:rel1} still holds and implies that
$q_h|_\edge^{}=\alpha$ with $\alpha$ given by \eqref{eq:rel2}. Then, like
before, it follows from \eqref{EQ:hdg_flux} that
\begin{equation}\label{eq:rel5}
  {\tau_E} u_h(a) = {\tau_E} \lambda_h(a) + \int_\edge f \varphi_a \dx,\qquad
  {\tau_E} u_h(b) = {\tau_E} \lambda_h(b) + \int_\edge f \varphi_b \dx.
\end{equation}
Therefore, the relations \eqref{eq:fluxes} again hold.

Plugging the relations \eqref{eq:fluxes} into \eqref{EQ:hdg_global_equiv}
yields the result.
\end{proof}

Theorem~\ref{TH:hdg_is_fd} implies that the HDG method on graphs can be interpreted as a finite difference method. However, this finite difference method is equipped with physically meaningful fluxes, and is locally conservative if interpreted as an HDG scheme. Moreover, e.g. symmetry of the diffusion equation is preserved in the discrete level \cite{CockburnGL09,RuppGK22}. Thus, it is \emph{mimetic} in the sense of \cite{HymanMSS02}.

\begin{theorem}\label{TH:hdg_is_fd2}
Let $\Omega$, $\kappa$, $\tau_E$, and $M(\node)$ be like in Theorem~\ref{TH:hdg_is_fd}
and let $U(\edge)=\polynomialsp_0(\edge)$ and $Q(\edge)=\polynomialsp_0(\edge)$
or $Q(\edge)=\polynomialsp_1(\edge)$ for all $\edge\in\setEdge$. Then
$\lambda_h\in M_h$ is a solution of the HDG method \eqref{EQ:hdg_scheme} if and
only if $\lambda_\node = g_\textup D(\node)$ for
$\node\in\setNode\cap\Gamma_\textup D$ and $\lambda_h$ satisfies the finite
difference equations
\begin{equation}\label{EQ:fd_system_tau}
  \sum_{\edge\in\setEdge,\,\partial E\ni\node}\,
  \left(\frac{\kappa_\edge}{|\edge|}+\frac{\tau_E}2\right)
  (\lambda_\Node - \lambda_{\node'}) 
  = \frac12\,\sum_{\edge\in\setEdge,\,\partial E\ni\node}\, 
  \int_\edge f\dx
\end{equation}
for every $\node \in \setNode$ with $\node \not \in \Gamma_\textup D$, where we
use the notation introduced in Theorem~\ref{TH:hdg_is_fd}.
\end{theorem}
\begin{proof}
The proof follows the lines of the proof of Theorem~\ref{TH:hdg_is_fd}. We
again start with the case $Q(\edge)=P_1(\edge)$. Like before, testing 
\eqref{EQ:hdg_flux} with $v_h = 1$ yields
\begin{equation}\label{eq:rel3}
  2\tau_E\,u_h|_\edge^{} = 
  \tau_E (\lambda_h(a) + \lambda_h(b)) +\int_\edge f \dx - \int_\edge q_h^\prime \dx
\end{equation}
so that \eqref{eq:rel1} again holds and hence $q_h|_E^{}$ is the same as in the
proof of Theorem~\ref{TH:hdg_is_fd}. Thus, using \eqref{eq:rel3}, it follows
that
\begin{align*}
  (\vec q_h\cdot\Normal_E + {\tau_E} (u_h- \lambda_h))(a) & = 
  \left(\frac{{\kappa_E}}{|\edge|}+\frac{\tau_E}2\right)
  (\lambda_h(b) - \lambda_h(a)) + \frac12\,\int_\edge f\dx, \\
  (\vec q_h\cdot\Normal_E + {\tau_E} (u_h - \lambda_h))(b) & = 
  \left(\frac{{\kappa_E}}{|\edge|}+\frac{\tau_E}2\right)
  (\lambda_h(a) - \lambda_h(b)) + \frac12\,\int_\edge f\dx
\end{align*}
and the same relations also hold $Q(\edge)=\polynomialsp_0(\edge)$. Thus, 
\eqref{EQ:hdg_global_equiv} implies \eqref{EQ:fd_system_tau}.
\end{proof}

\begin{remark}
If the graph considered in the above theorems represents a division of an
interval and $\kappa$ and $f$ are constant, then we observe that, under the
assumptions of Theorem~\ref{TH:hdg_is_fd}, the HDG solution $\lambda_h$
coincides with the standard finite difference solution of the equation
$-\kappa u^{\prime\prime}=f$. Under the assumptions of
Theorem~\ref{TH:hdg_is_fd2}, the solution $\lambda_h$ is different. In
particular, if $\tau_\edge=\tau$ for any $\edge\in\setEdge$ and the division is
equidistant with mesh width $h$, then $\lambda_h$ is the standard finite
difference solution of the equation $-(\kappa+\tau h/2)u^{\prime\prime}=f$.
\end{remark}

\begin{lemma}
Under the assumptions of Theorems~\ref{TH:hdg_is_fd} and~\ref{TH:hdg_is_fd2},
the finite difference systems \eqref{EQ:fd_system} and
\eqref{EQ:fd_system_tau}, respectively, are well-posed and their solutions
attain only nonnegative values. Moreover, if
$U(\edge)=\polynomialsp_0(\edge)$ or $Q(\edge)=\polynomialsp_0(\edge)$ for any
$\edge\in\setEdge$, then the solution $u_h$ of the HDG method
\eqref{EQ:hdg_scheme} satisfies $u_h\ge0$. If both
$U(\edge)=\polynomialsp_1(\edge)$ and $Q(\edge)=\polynomialsp_1(\edge)$, then
in general only the mean value $u_{0,\edge}$ introduced in
Theorem~\ref{TH:main_result} is nonnegative.
\end{lemma}
\begin{proof}
The well-posedness of \eqref{EQ:fd_system} follows from
Theorem~\ref{thm:hdg_solvability}. Since the matrix corresponding to the
left-hand side of \eqref{EQ:fd_system} has nonpositive off-diagonal entries and
zero row sums, it is a matrix of nonnegative type. Thus, completing
\eqref{EQ:fd_system} by the Dirichlet boundary conditions, the resulting square
system matrix is an M-matrix, see \cite[Corollary 3.13]{BarrenecheaJK24},
which implies the result for $\lambda_h$. If
$U(\edge)=\polynomialsp_0(\edge)$ or $Q(\edge)=\polynomialsp_0(\edge)$ for any
$\edge\in\setEdge$, then $u_h\ge0$ in view of \eqref{eq:rel5}, \eqref{eq:rel3},
and the formula \eqref{eq:rel4} implying that 
$\int_\edge q_h^\prime\dx\le\int_\edge f\dx$. If $U(\edge)=\polynomialsp_1(\edge)$
and $Q(\edge)=\polynomialsp_1(\edge)$, then the result for $u_{0,\edge}$
follows from taking the mean of the equations in \eqref{EQ:bdr_vals}.
\end{proof}

\begin{lemma}
Let the assumptions of Theorem \ref{TH:hdg_is_fd} hold,
$U(\edge)=\polynomialsp_1(\edge)$ and $Q(\edge)=\polynomialsp_1(\edge)$ for
all $\edge\in\setEdge$ and $f$ be constant on each $\edge \in \setEdge$. Then
the solution $u_h$ of the HDG method \eqref{EQ:hdg_scheme} satisfies
$u_h\ge0$.
\end{lemma}
\begin{proof}
We know that $\lambda_h \ge 0$ and that $u_h$ is a piecewise linear function.
Let us use the notation of the proof of Theorem \ref{TH:hdg_is_fd}. Since
$f$ is constant on $E$, we have
$\int_\edge f\varphi_a\dx=\int_\edge f\varphi_b\dx=\int_\edge f\dx/2$, and it
follows from \eqref{EQ:bdr_vals} that $u_h(a) \ge 0$ and $u_h(b) \ge 0$ so that
$u_h\ge0$ on $\edge$.
\end{proof}

% ---------------------------------------------------------------------------------------------------
\section{Positivity preservation of HDG methods on hypergraphs}\label{SEC:pos_pres_hg}
% ---------------------------------------------------------------------------------------------------

Let us investigate the positivity preservation of the HDG method
\eqref{EQ:hdg_scheme} for hypergraphs consisting of $d$-dimensional ($d>1$)
hyperedges.

% ---------------------------------------------------------------------------------------------------
\subsection{Piecewise constant approximations}\label{SEC:p0}
% ---------------------------------------------------------------------------------------------------
% 
In this section, we consider piecewise constant approximations, i.e., we set 
$U(\edge)=\polynomialsp_0(\edge)$, $\vec Q(\edge)=\polynomialsp_0(\edge)^d$,
and $M(\node)=\polynomialsp_0(\node)$. Testing \eqref{EQ:hdg_primary} with
$\vec p_h = \vec e_k$, $k=1,\dots,d$, and \eqref{EQ:hdg_flux} with $v_h = 1$
immediately yields for any polyhedron $\edge\in\setEdge$ that
\begin{subequations}\label{EQ:local_reconstr}
\begin{align}
   \vec q_h|_\edge^{} &= \vec q_E(\lambda_h) := 
   -\bar\kappa_E\int_{\partial \edge}\lambda_h\Normal_\edge\ds\quad
   \text{with}\quad\bar\kappa_E=\frac{1}{\int_\edge \kappa^{-1} \dbx},
   \label{EQ:local_reconstr_q}
   \\
   u_h|_\edge^{}&= 
   u_E(\lambda_h) + \frac{1}{\tau_E|\partial \edge|} \int_\edge f \dbx
   \quad\text{with}\quad u_E(\lambda_h)=
   \frac{1}{|\partial\edge|} \int_{\partial \edge} \lambda_h \ds.
   \label{eq:uh_const}
\end{align}
\end{subequations}
Then \eqref{EQ:hdg_global} can be equivalently written in the form
\begin{equation}\label{eq:condensed}
   a_h(\lambda_h,\mu_h)=
   \sum_{\edge\in\setEdge}\,\frac{1}{|\partial \edge|} 
   \int_{\partial\Edge}\mu_h\ds\int_\edge f\dbx\qquad\forall\mu_h\in M_h^0
\end{equation}
with
\begin{equation}\label{eq:ah_pw_const}
  a_h(\lambda_h,\mu_h)=-\sum_{\edge\in\setEdge}\,
  \int_{\partial\Edge}\mu_h\,(\vec q_E(\lambda_h)\cdot\Normal_E 
  + \tau_E(u_E(\lambda_h) -\lambda_h))\ds.
\end{equation}
Denoting by $\mu_\Node\colon \Sigma_h \to \{0,1\}$ the characteristic function
of a hypernode $\node\in\setNode$ and setting 
$\lambda_\node=\lambda_h|_\node^{}$, the problem \eqref{eq:condensed} is
equivalent to the linear system
\begin{equation}\label{eq:condensed_system}
   \sum_{\node'\in\setNode}a_{\node\node'}\,\lambda_{\node'}
   =\sum_{\edge\in\setEdge,\,\partial\edge\supset\node}\,
   \frac{|N|}{|\partial \edge|}\int_\edge f\dbx\qquad
   \forall\node\in\setNode^0,
\end{equation}
where $a_{\node\node'}=a_h(\mu_{\Node'},\mu_\Node)$ and
\begin{equation*}
   \setNode^0 = \{\node\in\setNode\colon\node\not\subset\Gamma_\textup D\}.
\end{equation*}
Since $a_h(1,\mu_h)=0$ for any $\mu_h\in M_h^0$, the row sums of the matrix in
\eqref{eq:condensed_system} vanish. Moreover,
\begin{equation*}
   a_{\node\node'}=\sum_{\edge\in\setEdge,\,\partial\edge\supset\node,\node'}\,
   |\node|\,|\node'|
   \left(\bar\kappa_E\Normal_E|_\node^{}\cdot\Normal_E|_{\node'}^{}
   -\frac{\tau_\edge}{|\partial\edge|}\right)\quad
   \forall\node\in\setNode^0,\,\node'\in\setNode,\,\node\neq\node'.
\end{equation*}
If there is no $\edge\in\setEdge$ whose boundary contains both $\node$ and
$\node'$, then $a_{\node\node'}=0$. Thus, if
\begin{equation}\label{eq:cond1}
   \Normal_E|_\node^{}\cdot\Normal_E|_{\node'}^{}\le0\qquad
   \forall\edge\in\setEdge,\,\node,\node'\subset\partial\edge,\,\node\neq\node'
\end{equation}
or
\begin{equation}\label{eq:cond2}
   \tau_\edge\ge\bar\kappa_E|\partial\edge|\qquad\forall\edge\in\setEdge,
\end{equation}
then the off-diagonal entries of the matrix in \eqref{eq:condensed_system} are
nonpositive and hence the matrix is of nonnegative type. Now, completing
\eqref{eq:condensed_system} by the Dirichlet boundary condition 
$\lambda_h=g_{{\textup D},h}$ on $\Gamma_\textup D$, we obtain a system with a
square matrix of nonnegative type which is nonsingular due to
Theorem~\ref{thm:hdg_solvability}. Consequently, this square system matrix is
an M-matrix, see \cite[Corollary 3.13]{BarrenecheaJK24}, which implies that
$\lambda_h\ge0$ since the right-hand side of the linear system is nonnegative.
Then $u_h\ge0$ according to \eqref{eq:uh_const}.

Note that condition \eqref{eq:cond1} restricts the shapes of the hyperedges
$E$. For example, it is satisfied for nonobtuse $d$-simplices or for 
$d$-rectangles. On the other hand, condition \eqref{eq:cond2} requires that
$\tau_E\sim\kappa_\edge/\operatorname{diam}(\edge)$ with
$\kappa_E\approx\kappa|_\edge^{}$.

% ---------------------------------------------------------------------------------------------------
\subsection{Lowest-order Raviart--Thomas spaces on simplices}
% ---------------------------------------------------------------------------------------------------
% 
In this section, we assume that all hyperedges are simplices. We consider 
the same spaces $U(\edge)=\polynomialsp_0(\edge)$ and
$M(\node)=\polynomialsp_0(\node)$ as in the previous section but we enlarge the
space $\vec Q(\edge)$ to $\vec Q(\edge)=\polynomialsRT_0(\edge)$, which
allows us to set $\tau_\edge\equiv0$. Since the non-hybrid Raviart--Thomas
method (RT) and its hybridized version (RT-H) produce the same 
solutions $u_h$ and $\vec q_h$, see Theorem 7.2.1 in \cite{boffi2013mixed}, the below analysis provides the positivity
preservation also for the RT case.

Assuming that $\kappa$ takes a constant value $\kappa_\edge>0$ on each
$\edge\in\setEdge$, the solution of the HDG method \eqref{EQ:hdg_scheme} with
$\tau_\edge=0$ satisfies
\begin{subequations}
\begin{align}
 \vec q_h|_\edge^{}(\vec x) & = \vec q_E(\lambda_h) + \frac1{d |\edge|}
\int_\edge f \dbx \; (\vec x - \vec x_\edge), \\
 u_h|_\edge^{} & = \frac{\kappa_\edge^{-1}}{d^2 |\edge|^2} \int_\edge f \dbx
\int_\edge \|\vec x - \vec x_\edge \|^2 \dbx + \frac1{d |\edge|} \int_{\partial \edge}
\lambda_h \underbrace{ (\vec x - \vec x_\edge) \cdot \Normal_\edge }_{ \ge 0 } \ds,
\label{eq:uh_RT}
\end{align}
\end{subequations}
where $\vec q_E(\lambda_h)$ is defined by \eqref{EQ:local_reconstr_q} and
$\vec x_\edge$ denotes the barycenter of $\edge$. Then \eqref{EQ:hdg_global}
can be equivalently written in the form
\begin{equation}\label{eq:condensedRT}
   a_h(\lambda_h,\mu_h)=
   \sum_{\edge\in\setEdge}\,\frac1{d|\edge|}\int_{\partial\Edge}\mu_h
   (\vec x-\vec x_\edge)\cdot\Normal_\edge\ds\int_\edge f\dbx\qquad
   \forall\mu_h\in M_h^0,
\end{equation}
where $a_h$ is defined by \eqref{eq:ah_pw_const} with $\tau_\edge=0$. Denoting
by $v$ the height of a simplex $\edge$ on its face $\node$, one has 
$(\vec x-\vec x_\edge)\cdot\Normal_\edge=v/(d+1)$ for any $\vec x\in\node$ and
hence
\begin{equation}\label{eq:identity}
   \int_\node(\vec x-\vec x_\edge)\cdot\Normal_\edge\ds=\frac{|\node|v}{d+1}
   =\frac{d|\edge|}{d+1}.
\end{equation}
Thus, using the notation of Section~\ref{SEC:p0}, the problem
\eqref{eq:condensedRT} is equivalent to the linear system
\begin{equation}\label{eq:condensed_systemRT}
   \sum_{\node'\in\setNode}a_{\node\node'}\,\lambda_{\node'}
   =\frac1{d+1}\sum_{\edge\in\setEdge,\,\partial\edge\supset\node}\,
   \int_\edge f\dbx\qquad\forall\node\in\setNode^0.
\end{equation}
Therefore, if the mesh assumption \eqref{eq:cond1} holds, it follows in the
same way as in Section~\ref{SEC:p0} that $\lambda_h\ge0$, which implies
$u_h\ge0$ in view of \eqref{eq:uh_RT}.

\begin{remark}[RT-H method with \(\tau_\edge > 0\)]
If the HDG method \eqref{EQ:hdg_scheme} is considered with \(\tau_\edge\ge0\),
then, denoting
\[ A = \kappa_\edge^{-1} \int_\edge \|\vec x - \vec x_\edge \|^2 \dbx 
\qquad \text{ and } \qquad B = d^2 |\edge|^2 + \tau_\edge |\partial \edge| A, \]
we can determine the local reconstruction of the solution as
 \begin{align*}
  \vec q_h|_\edge^{}(\vec x) & =  \frac{\vec x - \vec x_\edge}{B} \left[ d |\edge|
\tau_\edge \int_{\partial \edge} \lambda_h \ds - \tau_\edge | \partial \edge |
\int_{\partial \edge} \lambda_h (\vec x - \vec x_\edge) \cdot \Normal_\edge
\dbx + d |\edge| \int_\edge f \dbx \right] \\
  & \qquad - \frac{\kappa_\edge}{|\edge|} \int_{\partial \edge} \lambda_h
\Normal_\edge \ds ,\\
  u_h |_\edge^{} & = \frac{1}{B} \left[ A \tau_\edge \int_{\partial \edge}
\lambda_h \ds + d |E| \int_{\partial \edge} \lambda_h (\vec x - \vec x_\edge)
\cdot \Normal_\edge \ds + A \int_\edge f \dbx \right].
 \end{align*}
 This can be verified by plugging these representations into
\eqref{EQ:hdg_scheme}. Applying \eqref{eq:identity}, we derive that the matrix
entries in an analogue of \eqref{eq:condensed_systemRT} are given by
 \[ a_{\node,\node'} = \sum_{\substack{\edge \in \setEdge\\\partial \edge
\supset \node,\node'}} \left( \left[ \frac{\kappa_\edge(
\Normal_\edge|_\node^{} \cdot \Normal_\edge|_{\node'}^{} )}{|\edge|} - \frac{\tau_\edge^2 A}{B} \right] |\node| |\node'| - \tau_\edge \frac{d^2 |\edge|^2}{(d+1)B} \left[ |\node| + |\node'| - \frac{|\partial \edge|}{d+1} \right] \right) \]
and hence we can ensure that \( a_{\node,\node'} \le 0 \) for $\node\neq\node'$
if \(\tau_\edge > 0\) is sufficiently large.
\end{remark}

% ---------------------------------------------------------------------------------------------------
\subsection{HDG methods on rectangular hypergraphs for $\tau_E>0$,
$U(\edge)=\polynomialsp_0(\edge)$, and $M(\node)=\polynomialsp_0(\node)$}
\label{sec:p0_1_p0}
% ---------------------------------------------------------------------------------------------------

Let us start with $\vec Q(\edge)=\polynomials_1(\edge)^2$. Assuming that
$\kappa$ takes a constant value $\kappa_\edge>0$ on each $\edge\in\setEdge$, it follows from Remark~\ref{rem:coefficients} that the solution of the HDG
method \eqref{EQ:hdg_scheme} is given by \eqref{eq:uh_qh_general} with
$a_1=a_2=b=a^1_2=b^1=a^2_1=b^2=0$, with $\vec q_0$ given by \eqref{eq:q0} and
with $u_0$, $a^1_1$ and $a^2_2$ satisfying \eqref{eq:u0} and
\begin{equation}\label{eq:aii}
  a^i_i  = \frac{12\kappa_\edge}{h_i^2} u_0 
  -\frac{6\kappa_\edge}{|\edge|\,h_i}\int_{\node_i^+\cup\node_i^-}\lambda_h\ds,
  \qquad i=1,2.
\end{equation}
Thus, for any $\edge\in\setEdge$, one has 
$\vec q_h|_\edge^{}\in\polynomialsp_1(\edge)^2$ and
even $\vec q_h|_\edge^{}\in\polynomialsRT_0(\edge)$ (rectangular version). Therefore, the HDG
method with $\vec Q(\edge)=\polynomials_1(\edge)^2$ is equivalent to the HDG
method using $\vec Q(\edge)=\polynomialsp_1(\edge)^2$ or 
$\vec Q(\edge)=\polynomialsRT_0(\edge)$.

Denoting
\begin{equation*}
  A=12\kappa_\edge\left(h_1\int_{\node_1^+\cup\node_1^-}\lambda_h\ds
  -h_2\int_{\node_2^+\cup\node_2^-}\lambda_h\ds\right),\qquad
  B=12\kappa_\edge\int_\edge f\dbx,
\end{equation*}
and
\begin{equation*}
  C=\tau_\edge|\edge|\,|\partial\edge|+12\kappa_\edge(h_1^2+h_2^2),
\end{equation*}
the solution of \eqref{eq:u0} and \eqref{eq:aii} is given by
\begin{align}
   a^1_1&=-\frac{\tau_\edge h_2+6\kappa_\edge}{|\edge|C}A
   +\frac{h_2^2}{|\edge|C}B,\qquad
   a^2_2=\frac{\tau_\edge h_1+6\kappa_\edge}{|\edge|C}A
   +\frac{h_1^2}{|\edge|C}B,\label{eq:a11_a22}\\
   u_0&= \frac1{|\partial\edge|}\int_{\partial\edge} \lambda_h\ds +
   \frac1{\tau_\edge|\partial\edge|}\int_\edge f \dbx
   -\frac{\tau_E(h_1-h_2)A+(h_1^2+h_2^2)B}{\tau_\edge|\partial\edge|C}.
   \label{eq:u0_def}
\end{align}
To see this, we first rewrite the formula for $u^0$ in the form
\begin{equation}\label{eq:u0_pos_formula}
   \frac{C}{|E|}u_0=\tau_\edge\int_{\partial\edge} \lambda_h\ds
   +\frac{6\kappa_\edge}{h_1}\int_{\node_1^+\cup\node_1^-}\lambda_h\ds
   +\frac{6\kappa_\edge}{h_2}\int_{\node_2^+\cup\node_2^-}\lambda_h\ds
   +\int_\edge f \dbx.
\end{equation}
Since
\begin{equation*}
   \frac1{h_1}\left(\tau_\edge+\frac{6\kappa_\edge}{h_1}\right)
  +\frac1{h_2}\left(\tau_\edge+\frac{6\kappa_\edge}{h_2}\right)
  =\frac{C}{2|\edge|^2},
\end{equation*}
one obtains, for any $i\neq j$, $i,j\in\{1,2\}$,
\begin{equation*}
   12\kappa_\edge u_0=
   (-1)^i\,h_i^2\frac{\tau_\edge h_j+6\kappa_\edge}{|\edge|C}A
   +\frac{6\kappa_\edge}{|\edge|}h_i
   \int_{\node_i^+\cup\node_i^-}\lambda_h\ds+\frac{|\edge|}C B.
\end{equation*}
Using this relation and \eqref{eq:a11_a22}, the validity of \eqref{eq:u0} and
\eqref{eq:aii} easily follows.

It remains to investigate the consequences of the coupling conditions
\eqref{EQ:hdg_global}. Without loss of generality, we can consider the edge 
$\node_1^+$ of the rectangle $\edge$. Then we have
\begin{equation*}
  (\vec q_h|_E^{}\cdot\Normal_E + \tau_E(u_h|_E^{}-\lambda_h))|_{N_1^+}^{}
  =\vec q_0\cdot\vec e_1+a_1^1\frac{h_1}2+\tau_E u_0-\tau_E\lambda_h|_{N_1^+}^{}
\end{equation*}
and since
\begin{equation*}
   \frac{\tau_\edge h_2+6\kappa_\edge}{|\edge|C}\frac{h_1}2
   +\frac{\tau_E(h_1-h_2)}{|\partial\edge|C}
   =\frac{2\tau_\edge|\edge|+3\kappa_E|\partial\edge|}
   {|\edge|\,|\partial\edge|C}h_1,
\end{equation*}
it follows from \eqref{eq:a11_a22} and \eqref{eq:u0_def} that
\begin{align*}
  &(\vec q_h|_E^{}\cdot\Normal_E + \tau_E(u_h|_E^{}-\lambda_h))|_{N_1^+}^{}
  =- \frac{\kappa_\edge}{|\edge|} \int_{N_1^+} \lambda_h\ds
   -\frac{\tau_\edge(2h_1+h_2)}{|\partial\edge|}\lambda_h|_{N_1^+}^{}\\
& +\frac{\tau_\edge}{|\partial\edge|}\int_{N_2^+\cup N_2^-} \lambda_h\ds 
   -\frac{2\tau_\edge|\edge|+3\kappa_E|\partial\edge|}
   {|\edge|\,|\partial\edge|C}h_1A\\
&  + \left(\frac{\kappa_\edge}{|\edge|} 
  +\frac{\tau_\edge}{|\partial\edge|}\right)\int_{N_1^-} \lambda_h\ds 
   +\frac{\tau_\edge|\edge|+6\kappa_\edge h_2}{C}\int_\edge f \dbx.
\end{align*}
To prove the nonnegativity of $\lambda_h$, it is needed that, in the above
expression, the sums of the coefficients at $\lambda_h|_{N_1^-}^{}$ and
$\lambda_h|_{N_2^\pm}^{}$ are nonnegative. This is obvious for
$\lambda_h|_{N_2^\pm}^{}$. The terms containing $\lambda_h|_{N_1^-}^{}$ sum to
\begin{equation}\label{eq:coeff_at_l1m}
   \left(\frac{\kappa_\edge}{|\edge|}+\frac{\tau_\edge}{|\partial\edge|}
   - 12\kappa_\edge h_1^2\,\frac{2\tau_\edge|\edge|+3\kappa_E|\partial\edge|}
   {|\edge|\,|\partial\edge|C}\right)\int_{N_1^-} \lambda_h\ds.
\end{equation}
Denoting by $\varrho$ the aspect ration of $\edge$ and setting
\begin{equation*}
   s=\frac{\tau_\edge\,|\edge|}{\kappa_\edge\,|\partial\edge|},
\end{equation*}
the nonnegativity of the coefficient in \eqref{eq:coeff_at_l1m} is guaranteed
if
\begin{equation}\label{eq:quadr_ineq}
   (s+1)(s|\partial\edge|^2+12(h_1^2+h_2^2))\ge12|\edge|\varrho(2s+3).
\end{equation}
Note that this quadratic inequality is not satisfied for $s=0$. The smallest
positive solution $s_0$ of \eqref{eq:quadr_ineq} depends only on the aspect
ratio $\varrho$. Thus, the coefficient in \eqref{eq:coeff_at_l1m} is
nonnegative if $\tau_E\ge C(\varrho)\kappa_\edge/\operatorname{diam}(\edge)$
for some positive function $C(\varrho)$ of $\varrho$. If $\edge$ is a square,
i.e., $h_1=h_2=:h$, then the inequality \eqref{eq:quadr_ineq} reduces to
$2(s+1)\ge3$ and hence the coefficient in \eqref{eq:coeff_at_l1m} is
nonnegative if $\tau_\edge$ satisfies \eqref{eq:q0_p0_p0_cond_square}.

The above considerations show that, under the condition \eqref{eq:quadr_ineq},
the constraint \eqref{EQ:hdg_global} can be written in the form of the linear
system \eqref{eq:condensed_system} with a different but nonnegative right-hand
side and with $a_{\node\node'}\le0$ for $\node\neq\node'$. Setting
$\lambda_h=1$ on $\Sigma_h$ and $f=0$, we obtain $(\vec q_h|_E^{}\cdot\Normal_E
+ \tau_E(u_h|_E^{}-\lambda_h))|_{\partial\edge}^{}=0$ for any
$\edge\in\setEdge$ since the above-used formulas give $\vec q_0=\vec0$, $A=B=0$, and $u_0=1$. This proves that the row
sums of the matrix in \eqref{eq:condensed_system} vanish. Now, one concludes
that $\lambda_h\ge0$ in the same way as in Section~\ref{SEC:p0} and $u_h\ge0$
follows from \eqref{eq:u0_pos_formula}.

\begin{remark}
For $\tau_\edge=0$, it follows from \eqref{eq:a11_a22} and
\eqref{eq:u0_pos_formula} that
\begin{align*}
 u_h|_\edge^{}(x_1,x_2) &=\frac{|E|}{(h_1^2+h_2^2)}\left(
    \frac1{2h_1}\int_{\node_1^+\cup\node_1^-}\lambda_h\ds
   +\frac1{2h_2}\int_{\node_2^+\cup\node_2^-}\lambda_h\ds
   +\frac1{12\kappa_\edge}\int_\edge f \dbx\right),\\
 \vec q_h|_\edge^{}(x_1,x_2) & = \vec q_0 + 
 \frac1{2|\edge|(h_1^2+h_2^2)}A
 \begin{pmatrix}
   -x_1\\[1mm]
    x_2
 \end{pmatrix}
 +\frac1{|\edge|(h_1^2+h_2^2)}\int_\edge f\dbx
 \begin{pmatrix}
   h_2^2 x_1\\[1mm]
   h_1^2 x_2
 \end{pmatrix}.
\end{align*}
These formulas again represent a solution of \eqref{EQ:hdg_primary} and
\eqref{EQ:hdg_flux}, although \eqref{eq:a11_a22} and \eqref{eq:u0_pos_formula}
were derived for $\tau_\edge>0$. This is not surprising since, for 
$\vec Q(\edge)=\polynomialsRT_0(\edge)$, the HDG method \eqref{EQ:hdg_scheme}
is well-posed also for $\tau_\edge=0$. However, the coefficient in
\eqref{eq:coeff_at_l1m} is negative if $h_2<\sqrt2 h_1$ and hence the matrix of
\eqref{eq:condensed_system} is not of nonnegative type. Hence the positivity
preservation of the HDG method cannot be expected, see the counterexample in
Section~\ref{sec:counter_RT}.
\end{remark}

% ---------------------------------------------------------------------------------------------------
\section{Counterexamples in two dimensions}\label{SEC:counter}
% ---------------------------------------------------------------------------------------------------
% 
In this section we consider the two-dimensional case and present several 
counterexamples demonstrating that, for many choices of the spaces $U(\edge)$,
$\vec Q(\edge)$, and $M(\node)$, the HDG method \eqref{EQ:hdg_scheme} is not
positivity preserving in general. It will be always assumed that
\begin{equation}\label{eq:counterex_assumptions}
   \kappa=1,\qquad f=0,\qquad
   \mbox{$\Gamma_\textup D$ consists of all boundary hypernodes.}
\end{equation}
We start with the natural choice 
$U(\edge)=\polynomialsp_1(\edge)$, $\vec Q(\edge)=\polynomialsp_1(\edge)^2$, 
$M(\node)=\polynomialsp_1(\node)$.

\subsection{Counterexample for $\polynomialsp_1(\setEdge)\times
\polynomialsp_1(\setEdge)^2\times\polynomialsp_1(\setNode)$}
\label{sec:counterex_P1_P1_P1}
Let us first consider the case when $\setEdge$ consists of only one
hyperedge $E$ being a rectangle with edges parallel to the coordinate axes.
Assume \eqref{eq:counterex_assumptions}, i.e., $\Gamma_\textup D=\partial E$,
and let
\begin{equation}\label{eq:rel8}
 g_\textup D (x,y) = (a(x-x_E) + b)(c(y-y_E) + d) \ge 0 \qquad 
 \text{ for all } (x,y) \in \overline\edge,
\end{equation}
where $(x_E,y_E)$ is the barycenter of $E$ and $a$, $b$, $c$, $d$ are real
numbers. Consider the HDG method \eqref{EQ:hdg_scheme} with spaces defined using
$U(\edge)=\polynomialsp_1(\edge)$, $\vec Q(\edge)=\polynomialsp_1(\edge)^2$,
and $M(\node)=\polynomialsp_1(\node)$. Then we can use
$g_{{\textup D},h}=g_\textup D|_{\partial\edge}^{}$. In this case, the HDG
method \eqref{EQ:hdg_scheme} reduces to the equations \eqref{EQ:hdg_primary}
and \eqref{EQ:hdg_flux} since $\lambda_h$ is fully determined by the Dirichlet
boundary condition, i.e., $\lambda_h = g_\textup D|_{\partial \edge}^{}$. The
solution of the HDG method is given by
\begin{equation*}
  u_h(x,y) = ad(x-x_E) + bc(y-y_E) + bd, \qquad 
  \vec q_h = - \nabla g_\textup D.
\end{equation*}
To see this, note first that $g_\textup D-u_h=ac(x-x_E)(y-y_E)$, which
implies that
\begin{equation}\label{eq:rel6}
  \int_\edge u_h \dbx = \int_\edge g_\textup D \dbx,
\end{equation}
and
\begin{equation}\label{eq:rel7}
  \int_{\partial \edge} u_h v_h \ds = \int_{\partial \edge} g_\textup D v_h \ds
  \quad\forall v_h\in P_1(E).
\end{equation}
Using \eqref{eq:rel6}, we obtain
\begin{multline*}
  \int_\edge\vec q_h\cdot\vec p_h\dbx - \int_\edge u_h \div \vec p_h \dbx =
 -\int_\edge\nabla g_\textup D\cdot\vec p_h + g_\textup D \div \vec p_h\dbx\\
  = - \int_\edge \div ( g_\textup D \vec p_h) \dbx 
  = - \int_{\partial \edge} \lambda_h \vec p_h \cdot \Normal_E \ds,
\end{multline*}
for all $\vec p_h \in \vec Q_h$, proving the validity of
\eqref{EQ:hdg_primary}. Moreover, since $\nabla\cdot\vec q_h = - \Delta
g_\textup D = 0$, it follows from \eqref{eq:rel7} that \eqref{EQ:hdg_flux} is
satisfied as well. However, denoting by $h_x$ and $h_y$ the width and the
height of $E$, respectively, and choosing $a=c=2$, $b=h_x$ and $d=h_y$, one has
\eqref{eq:rel8} but $u_h(x_E-h_x/2,y_E-h_y/2) = -h_xh_y$, which violates the
bound $u_h \ge 0$, so that the HDG method is not positivity preserving.

The above counterexample can be extended to sets $\setEdge$ consisting of
several rectangles. As an example, let us consider two unit squares $E_1$,
$E_2$ sharing a common vertical edge $\node_{12}$ so that
\begin{equation*}
   x_{E_2}=x_{E_1}+1,\quad y_{E_2}=y_{E_1},\quad 
   \partial E_1\cap\partial E_2 = N_{12},
\end{equation*}
see Fig.~\ref{fig:two_squares}.
\begin{figure}[t]
\centering
\begin{tikzpicture}
   \draw[line width=0.4mm]
      (0,0) -- (5,0) -- (5,2.5) -- (0,2.5) -- cycle;
   \draw[line width=0.4mm, color=blue]
      (2.5,0) -- (2.5,2.5);
   \node {} (0.5,2.3) node[below] {$E_1$};
   \node {} (4.55,2.3) node[below] {$E_2$};
   \node {} (2.95,2.3) node[below] {\textcolor{blue}{$N_{12}$}};
   \node {} (1.25,1.25) node[below] {\small $(x_{E_1},y_{E_1})$};
   \node {} (3.75,1.25) node[below] {\small $(x_{E_2},y_{E_2})$};
   \fill[black] (1.25,1.25) circle (2pt);
   \fill[black] (3.75,1.25) circle (2pt);
   \pgfusepath{fill}
\end{tikzpicture}
\caption{Hypergraph consisting of two squares $E_1$, $E_2$.}
\label{fig:two_squares}
\end{figure}
Thus, $\setEdge=\{E_1,E_2\}$ and $\setNode$ consists of seven hypernodes. The
Dirichlet boundary $\Gamma_\textup D$ is again assumed to consist of all
boundary hypernodes so that the only hypernode which is not contained in
$\Gamma_\textup D$ is $\node_{12}$. Let us set
\begin{equation*}
  g_\textup D (x,y)=\left\{
    \begin{array}{cl}
       (2(x-x_{E_1}) + 1)(2(y-y_{E_1}) + 1)\quad
       &\text{ for all } (x,y) \in \overline{\edge_1},\\[1mm]
       (2(x-x_{E_2}) + 3)(2(y-y_{E_2}) + 1)\quad
       &\text{ for all } (x,y) \in \overline{\edge_2}\setminus\node_{12}.\\
    \end{array}\right.
\end{equation*}
Then $g_\textup D\in C(\overline{E_1\cup E_2})$ and $g_\textup D\ge0$. The
definition of $g_\textup D$ makes it possible to set
\begin{equation*}
    g_{{\textup D},h}=g_\textup D|_{\Gamma_\textup D}^{},\qquad 
    \lambda_h=g_\textup D|_{\Sigma_h}^{},\qquad
    \vec q_h|_{E_i}^{} = - \nabla g_\textup D|_{E_i}^{},\quad i=1,2.
\end{equation*}
Furthermore, we set
\begin{equation*}
  u_h(x,y)=\left\{
    \begin{array}{cl}
       2(x-x_{E_1}) + 2(y-y_{E_1}) + 1\quad
       &\text{ for all } (x,y) \in\edge_1,\\[1mm]
       2(x-x_{E_2}) + 6(y-y_{E_2}) + 3\quad
       &\text{ for all } (x,y) \in\edge_2.\\
    \end{array}\right.
\end{equation*}
Then, as it was shown above, the functions $u_h$, $\vec q_h$, $\lambda_h$
satisfy \eqref{EQ:hdg_primary} and \eqref{EQ:hdg_flux}. Moreover, if
$\tau_{E_1}=\tau_{E_2}=\tau$, we obtain
\begin{equation*}
  (\vec q_h|_{E_2}^{}\cdot\Normal_{E_2} 
  + \tau_{E_2}(u_h|_{E_2}^{}-\lambda_h))|_{\node_{12}}^{}
  =(4+2\tau)(y-y_{E_1})+2
  =-(\vec q_h|_{E_1}^{}\cdot\Normal_{E_1} 
  + \tau_{E_1}(u_h|_{E_1}^{}-\lambda_h))|_{\node_{12}}^{}
\end{equation*}
so that \eqref{EQ:hdg_global_equiv} and hence also \eqref{EQ:hdg_global} hold,
too. Consequently, the functions $u_h$, $\vec q_h$, $\lambda_h$ are the unique
solution of the HDG method \eqref{EQ:hdg_scheme}. Since 
$u_h(x_{E_1}-1/2,y_{E_1}-1/2) = -1$, the positivity preservation is again
violated.

\subsection{Counterexample for $\polynomials_1(\setEdge)\times
\polynomialsp_1(\setEdge)^2\times\polynomialsp_1(\setNode)$}
\label{sec:counterex_Q1_P1_P1}

Like in the previous counterexample, we consider a hypergraph consisting of the
two unit squares depicted in Fig.~\ref{fig:two_squares} and assume
\eqref{eq:counterex_assumptions}. Let us set
\begin{equation*}
  g_\textup D (x,y)=\left\{
    \begin{array}{cl}
       0\quad &\text{ for } x\le(x_{E_1}+x_{E_2})/2,\\[1mm]
       x-x_{E_2} + \frac12\quad
       &\text{ for } x>(x_{E_1}+x_{E_2})/2.\\
    \end{array}\right.
\end{equation*}
Then $g_\textup D$ is continuous and nonnegative. The values of $g_\textup D$
at the vertices of the hypergraph are shown in Fig.~\ref{fig:values} (left). We
\begin{figure}[t]
\centering
\begin{tikzpicture}
   \draw[line width=0.4mm, color=red]
      (0,0) -- (5,0) -- (5,2.5) -- (0,2.5) -- cycle;
   \draw[line width=0.6mm, color=blue]
      (2.5,0.15) -- (2.5,2.35);
   \draw[line width=0.4mm]
      (7,0) -- (12,0) -- (12,2.5) -- (7,2.5) -- cycle
      (9.5,0) -- (9.5,2.5);
   \node {} (2.5,1.9) node[right] {\textcolor{blue}{$\lambda_h=\lambda$}};
   \node {} (3.75,0) node[below] {\textcolor{red}{$\lambda_h=g_\textup D$}};
   \node {} (12,1.25) node[right] {$u_h$};
   \node {} (1.25,1.25) node[below] {\small $(x_{E_1},y_{E_1})$};
   \node {} (3.75,1.25) node[below] {\small $(x_{E_2},y_{E_2})$};
   \node {} (8.25,1.25) node[below] {\small $(x_{E_1},y_{E_1})$};
   \node {} (10.75,1.25) node[below] {\small $(x_{E_2},y_{E_2})$};
   \node {} (0,2.5) node[above left] {$0$};
   \node {} (2.5,2.5) node[above] {$0$};
   \node {} (5,2.5) node[above right] {$1$};
   \node {} (0,0) node[below left] {$0$};
   \node {} (2.5,0) node[below] {$0$};
   \node {} (5,0) node[below right] {$1$};
   \node {} (7,2.5) node[above left] {$-\frac\lambda8$};
   \node {} (9.5,2.5) node[above] {$\frac{5\lambda}8$};
   \node {} (12,2.5) node[above right] {$1-\frac\lambda8$};
   \node {} (7,0) node[below left] {$-\frac\lambda8$};
   \node {} (9.5,0) node[below] {$\frac{5\lambda}8$};
   \node {} (12,0) node[below right] {$1-\frac\lambda8$};
   \fill[black] (1.25,1.25) circle (2pt);
   \fill[black] (3.75,1.25) circle (2pt);
   \fill[black] (8.25,1.25) circle (2pt);
   \fill[black] (10.75,1.25) circle (2pt);
   \fill[black] (0,0) circle (2pt);
   \fill[black] (0,2.5) circle (2pt);
   \fill[black] (2.5,0) circle (2pt);
   \fill[black] (2.5,2.5) circle (2pt);
   \fill[black] (5,0) circle (2pt);
   \fill[black] (5,2.5) circle (2pt);
   \fill[black] (7,0) circle (2pt);
   \fill[black] (7,2.5) circle (2pt);
   \fill[black] (9.5,0) circle (2pt);
   \fill[black] (9.5,2.5) circle (2pt);
   \fill[black] (12,0) circle (2pt);
   \fill[black] (12,2.5) circle (2pt);
   \pgfusepath{fill}
\end{tikzpicture}
\caption{Counterexample for $\polynomials_1(\setEdge)\times
\polynomialsp_1(\setEdge)^2\times\polynomialsp_1(\setNode)$: values of
$\lambda_h$ (left) and $u_h$ (right).}
\label{fig:values}
\end{figure}
will consider the HDG method \eqref{EQ:hdg_scheme} with spaces defined using
$U(\edge)=\polynomials_1(\edge)$, $\vec Q(\edge)=\polynomialsp_1(\edge)^2$,
and $M(\node)=\polynomialsp_1(\node)$. Then we can again set
$g_{{\textup D},h}=g_\textup D|_{\Gamma_\textup D}^{}$. It can be expected that
the solution of the HDG method \eqref{EQ:hdg_scheme} will be constant in the
$y$-direction. Thus, in particular, $\lambda_h|_{\node_{12}}^{}=const=:\lambda$
so that $\lambda_h$ attains the values depicted in Fig.~\ref{fig:values}
(left). If $\lambda\neq0$, then $\lambda_h$ is discontinuous at the endpoints
of $\node_{12}$. We denote
\begin{equation*}
  u_h(x,y)=\left\{
    \begin{array}{cl}
       \frac{3\lambda}4(x-x_{E_1}) + \frac\lambda4\quad
       &\text{ for all } (x,y) \in\edge_1,\\[2mm]
       \left(1-\frac{3\lambda}4\right)(x-x_{E_2}) + \frac\lambda4+\frac12\quad
       &\text{ for all } (x,y) \in\edge_2.\\
    \end{array}\right.
\end{equation*}
The definition of this function is clearer from Fig.~\ref{fig:values} (right)
showing the values of $u_h$ at the vertices of the hypergraph. Note that $u_h$
is continuous across $\node_{12}$. Using the values shown in
Fig.~\ref{fig:values}, it is easy to verify that
\begin{equation}\label{eq:L2_proj_bndry}
  \int_{\partial \Edge_i} u_h v \ds = \int_{\partial \Edge_i} \lambda_h v \ds
  \quad\forall v\in Q_1(\Edge_i),\,\,i=1,2.
\end{equation}
Finally, we define $\vec q_h\in\vec Q_h$ by
\begin{equation*}
   \vec q_h|_{\edge_1}^{}(x,y)=\begin{pmatrix} 
   -\lambda-3\lambda(x-x_{E_1}) \\ 3\lambda(y-y_{E_1})\end{pmatrix},\qquad
   \vec q_h|_{\edge_2}^{}(x,y)=\begin{pmatrix} 
   \lambda-1-3\lambda(x-x_{E_2}) \\ 3\lambda(y-y_{E_2})\end{pmatrix}.
\end{equation*}
It can be verified that then \eqref{EQ:hdg_primary} is satisfied. Since
$\nabla\cdot\vec q_h|_{\edge_i}^{}=0$ for $i=1,2$, the validity of
\eqref{EQ:hdg_flux} follows from \eqref{eq:L2_proj_bndry}. It is easy to see
that
\begin{equation*}
  (\vec q_h|_{E_1}^{}\cdot\Normal_{E_1}
  +\tau_{E_1}(u_h-\lambda_h))|_{\node_{12}}^{}
  +(\vec q_h|_{E_2}^{}\cdot\Normal_{E_2} 
  + \tau_{E_2}(u_h-\lambda_h))|_{\node_{12}}^{}
  =1-5\lambda-(\tau_{E_1}+\tau_{E_2})\frac{3\lambda}8
\end{equation*}
and hence \eqref{EQ:hdg_global} holds if and only if
\begin{equation*}
  \lambda=\frac8{40+3(\tau_{E_1}+\tau_{E_2})}.
\end{equation*}
Then it is obvious from Fig.~\ref{fig:values} (right) that $u_h$ attains
negative values, which means that the HDG method is not positivity preserving.

\subsection{Counterexample for $\polynomialsp_1(\setEdge)\times
\polynomials_1(\setEdge)^2\times\polynomialsp_1(\setNode)$}
\label{sec:counterex_P1_Q1_P1}

In this section, we construct a counterexample for $\setEdge$ consisting of
only one hyperedge being the unit square $E=(-1/2,1/2)^2$. Assume
\eqref{eq:counterex_assumptions}, i.e., $\Gamma_\textup D=\partial E$, and let
\begin{equation*}
 g_\textup D (x,y) = \left(x+\frac12\right)\left(y+\frac12\right) \ge 0 \qquad 
 \text{ for all } (x,y) \in \overline\edge.
\end{equation*}
Consider the HDG method \eqref{EQ:hdg_scheme} with spaces defined using
$U(\edge)=\polynomialsp_1(\edge)$, $\vec Q(\edge)=\polynomials_1(\edge)^2$,
and $M(\node)=\polynomialsp_1(\node)$. Then we can again use
$g_{{\textup D},h}=g_\textup D|_{\partial\edge}^{}$. As already mentioned in
Section~\ref{sec:counterex_P1_P1_P1}, the HDG method \eqref{EQ:hdg_scheme}
reduces to the equations \eqref{EQ:hdg_primary} and \eqref{EQ:hdg_flux} in this
case. It can be verified that the solution of the HDG method is given by
\begin{equation*}
  u_h(x,y) = \frac{x+y}2+\frac14, \qquad 
  \vec q_h(x,y) = -\begin{pmatrix}y+\frac12\\[2mm]
  x+\frac12\end{pmatrix}.
\end{equation*}
Since $u_h(-1/2,-1/2)=-1/4$, this demonstrates that the HDG method is not
positivity preserving.

\subsection{Counterexample for $\polynomials_1(\setEdge)\times
\polynomials_1(\setEdge)^2\times\polynomialsp_1(\setNode)$}
\label{sec:counterex_Q1_Q1_P1}

Like in the counterexamples~\ref{sec:counterex_P1_P1_P1}
and~\ref{sec:counterex_Q1_P1_P1}, we consider a hypergraph consisting of the
two unit squares depicted in Fig.~\ref{fig:two_squares} and assume
\eqref{eq:counterex_assumptions}. Let us set
\begin{equation}\label{eq:gD_Q1_Q1_P1}
  g_\textup D (x,y)=\left\{
    \begin{array}{cl}
       -(x-x_{E_1} - \frac12)(y-y_{E_1} + \frac12)\quad
       &\text{ for } x\le(x_{E_1}+x_{E_2})/2,\\[1mm]
       0\quad &\text{ for } x>(x_{E_1}+x_{E_2})/2.
    \end{array}\right.
\end{equation}
Then $g_\textup D$ is continuous and nonnegative in
$\overline{\edge_1\cup\edge_2}$. The values of $g_\textup D$ at the vertices of
the hypergraph are shown in Fig.~\ref{fig:values_q1_q1} (left). We
\begin{figure}[t]
\centering
\begin{tikzpicture}
   \draw[line width=0.4mm, color=red]
      (0,0) -- (5,0) -- (5,2.5) -- (0,2.5) -- cycle;
   \draw[line width=0.6mm, color=blue]
      (2.5,0.15) -- (2.5,2.35);
   \draw[line width=0.4mm]
      (7.5,0) -- (12.5,0) -- (12.5,2.5) -- (7.5,2.5) -- cycle
      (10,0) -- (10,2.5);
   \node {} (2.5,2.2) node[right] {\textcolor{blue}{$\lambda_h=\lambda_1$}};
   \node {} (2.5,0.3) node[right] {\textcolor{blue}{$\lambda_h=\lambda_2$}};
   \node {} (3.75,0) node[below] {\textcolor{red}{$\lambda_h=g_\textup D$}};
   \node {} (12.5,1.25) node[right] {$u_h$};
   \node {} (1.25,1.25) node[below] {\small $(x_{E_1},y_{E_1})$};
   \node {} (3.75,1.25) node[below] {\small $(x_{E_2},y_{E_2})$};
   \node {} (8.75,1.25) node[below] {\small $(x_{E_1},y_{E_1})$};
   \node {} (11.25,1.25) node[below] {\small $(x_{E_2},y_{E_2})$};
   \node {} (0,2.5) node[above left] {$1$};
   \node {} (2.5,2.5) node[above] {$0$};
   \node {} (5,2.5) node[above right] {$0$};
   \node {} (0,0) node[below left] {$0$};
   \node {} (2.5,0) node[below] {$0$};
   \node {} (5,0) node[below right] {$0$};
   \node {} (7.6,2.5) node[above left] {$1+\alpha\lambda_1$};
   \node {} (10,2.5) node[above] {$\frac12\lambda_1-\alpha\lambda_2$};
   \node {} (12.45,2.5) node[above right] {$\alpha\lambda_1$};
   \node {} (7.6,0) node[below left] {$\alpha\lambda_2$};
   \node {} (10,0) node[below] {$\frac12\lambda_2-\alpha\lambda_1$};
   \node {} (12.45,0) node[below right] {$\alpha\lambda_2$};
   \fill[black] (1.25,1.25) circle (2pt);
   \fill[black] (3.75,1.25) circle (2pt);
   \fill[black] (8.75,1.25) circle (2pt);
   \fill[black] (11.25,1.25) circle (2pt);
   \fill[black] (0,0) circle (2pt);
   \fill[black] (0,2.5) circle (2pt);
   \fill[black] (2.5,0) circle (2pt);
   \fill[black] (2.5,2.5) circle (2pt);
   \fill[black] (5,0) circle (2pt);
   \fill[black] (5,2.5) circle (2pt);
   \fill[black] (7.5,0) circle (2pt);
   \fill[black] (7.5,2.5) circle (2pt);
   \fill[black] (10,0) circle (2pt);
   \fill[black] (10,2.5) circle (2pt);
   \fill[black] (12.5,0) circle (2pt);
   \fill[black] (12.5,2.5) circle (2pt);
   \pgfusepath{fill}
\end{tikzpicture}
\caption{Counterexample for $\polynomials_1(\setEdge)\times
\polynomials_1(\setEdge)^2\times\polynomialsp_1(\setNode)$: values of
$\lambda_h$ (left) and $u_h$ (right).}
\label{fig:values_q1_q1}
\end{figure}
will consider the HDG method \eqref{EQ:hdg_scheme} with spaces defined using
$U(\edge)=\polynomials_1(\edge)$, $\vec Q(\edge)=\polynomials_1(\edge)^2$,
and $M(\node)=\polynomialsp_1(\node)$. Then we can again set
$g_{{\textup D},h}=g_\textup D|_{\Gamma_\textup D}^{}$. The function
$g_{{\textup D},h}$ defines $\lambda_h$ at all hypernodes $\node\in\setNode$
except $\node_{12}$ where $\lambda_h$ is determined by its values $\lambda_1$,
$\lambda_2$ at the endpoints of $\node_{12}$, see~Fig.~\ref{fig:values_q1_q1}
(left). Precisely,
\begin{equation*}
   \lambda_h|_{\node_{12}}^{}(x,y)
   =\frac{\lambda_1+\lambda_2}2+(\lambda_1-\lambda_2)(y-y_{\edge_1}).
\end{equation*}
For simplicity, we consider the HDG method \eqref{EQ:hdg_scheme} with
$\tau_{\edge_1}=\tau_{\edge_2}=:\tau$ and we denote
\begin{equation*}
   \alpha=\frac{3-\tau}{8\tau+12},\qquad\beta=\frac{9\tau}{2\tau+3}.
\end{equation*}
Using this notation, we define a continuous function $u_h\in U_h\cap
C(\overline{\edge_1\cup\edge_2})$ such that, at the vertices of the hypergraph,
$u_h$ attains the values depicted in Fig.~\ref{fig:values_q1_q1} (right).
Furthermore, we define a function $\vec q_h\in\vec Q_h$ by
\begin{align*}
   \vec q_h|_{\edge_1}^{}(x,y)&=\begin{pmatrix*}[r]
   \frac12(1-\lambda_1-\lambda_2)-\frac32(\lambda_1+\lambda_2)(x-x_{E_1})
   +(1-\lambda_1+\lambda_2)(y-y_{E_1})\hspace*{12mm}\\
   -\beta(\lambda_1-\lambda_2)(x-x_{E_1})(y-y_{E_1})\\[2mm]
   -\frac12+x-x_{E_1}+\frac32(\lambda_1+\lambda_2)(y-y_{E_1})
   +\beta(\lambda_1+\lambda_2)(x-x_{E_1})(y-y_{E_1})
   \end{pmatrix*},\\[2mm]
   \vec q_h|_{\edge_2}^{}(x,y)&=\begin{pmatrix*}[r]
   \frac12(\lambda_1+\lambda_2)-\frac32(\lambda_1+\lambda_2)(x-x_{E_2})
   +(\lambda_1-\lambda_2)(y-y_{E_2})\hspace*{19mm}\\
   -\beta(\lambda_1-\lambda_2)(x-x_{E_2})(y-y_{E_2})\\[2mm]
   \frac32(\lambda_1+\lambda_2)(y-y_{E_2})
   -\beta(\lambda_1+\lambda_2)(x-x_{E_2})(y-y_{E_2})
   \end{pmatrix*}.
\end{align*}
It can be verified that then both \eqref{EQ:hdg_primary} and
\eqref{EQ:hdg_flux} are satisfied. Moreover, we obtain
\begin{align*}
  &(\vec q_h|_{E_1}^{}\cdot\Normal_{E_1}
  +\tau(u_h-\lambda_h))|_{\node_{12}}^{}
  +(\vec q_h|_{E_2}^{}\cdot\Normal_{E_2} 
  + \tau(u_h-\lambda_h))|_{\node_{12}}^{}\\
  &=\frac12-\frac52(\lambda_1+\lambda_2)
  +\left(1-(2+\beta)(\lambda_1-\lambda_2)\right)(y-y_{E_1})\\
  &\phantom{=}+\tau\left(-\left(\frac12+\alpha\right)(\lambda_1+\lambda_2)
  +(2\alpha-1)(\lambda_1-\lambda_2)(y-y_{E_1})\right)
\end{align*}
and hence \eqref{EQ:hdg_global} holds if and only if
\begin{align*}
  &1-5(\lambda_1+\lambda_2)-\tau(1+2\alpha)(\lambda_1+\lambda_2)=0,\\
  &1-(2+\beta)(\lambda_1-\lambda_2)-\tau(1-2\alpha)(\lambda_1-\lambda_2)=0,
\end{align*}
which gives
\begin{equation*}
  \lambda_1+\lambda_2=\frac{4\tau+6}{3\tau^2+29\tau+30},\qquad
  \lambda_1-\lambda_2=\frac{4\tau+6}{5\tau^2+29\tau+12}.
\end{equation*}
Thus, $\lambda_1>0$ but $\lambda_2\ge0$ if and only if
$5\tau^2+29\tau+12\ge3\tau^2+29\tau+30$, which is satisfied only for $\tau\ge3$.
Then it follows from Fig.~\ref{fig:values_q1_q1} (right) that $u_h$ is
nonnegative if and only if $\tau=3$ giving $\alpha=0$ and $\beta=3$.

Unfortunately, in general, $u_h$ and $\lambda_h$ may attain negative values
also for $\tau=3$. As an example, let us consider a hypergraph consisting of
three unit squares arranged as in Fig.~\ref{fig:values_q1_q1_3}.
\begin{figure}[t]
\centering
\begin{tikzpicture}
   \draw[line width=0.4mm]
      (0,0) -- (7.5,0) -- (7.5,2.5) -- (0,2.5) -- cycle;
   \draw[line width=0.6mm, color=blue]
      (2.5,0.15) -- (2.5,2.35) (5,0.15)   -- (5,2.35);
   \node {} (2.5,2.2) node[right] {\textcolor{blue}{$\lambda_h=\lambda_1$}};
   \node {} (2.5,0.3) node[right] {\textcolor{blue}{$\lambda_h=\lambda_2$}};
   \node {} (5,2.2) node[right] {\textcolor{blue}{$\lambda_h=\lambda_3$}};
   \node {} (5,0.3) node[right] {\textcolor{blue}{$\lambda_h=\lambda_4$}};
   \node {} (1.25,1.25) node[below] {\small $(x_{E_1},y_{E_1})$};
   \node {} (3.75,1.25) node[below] {\small $(x_{E_2},y_{E_2})$};
   \node {} (6.25,1.25) node[below] {\small $(x_{E_3},y_{E_3})$};
   \node {} (0,2.5) node[above left] {$1$};
   \node {} (2.5,2.5) node[above] {$\frac12\lambda_1$};
   \node {} (5,2.5) node[above] {$\frac12\lambda_3$};
   \node {} (7.5,2.5) node[above right] {$0$};
   \node {} (0,0) node[below left] {$0$};
   \node {} (2.5,0) node[below] {$\frac12\lambda_2$};
   \node {} (5,0) node[below] {$\frac12\lambda_4$};
   \node {} (7.5,0) node[below right] {$0$};
   \fill[black] (1.25,1.25) circle (2pt);
   \fill[black] (3.75,1.25) circle (2pt);
   \fill[black] (6.25,1.25) circle (2pt);
   \fill[black] (0,0) circle (2pt);
   \fill[black] (0,2.5) circle (2pt);
   \fill[black] (2.5,0) circle (2pt);
   \fill[black] (2.5,2.5) circle (2pt);
   \fill[black] (5,0) circle (2pt);
   \fill[black] (5,2.5) circle (2pt);
   \fill[black] (7.5,0) circle (2pt);
   \fill[black] (7.5,2.5) circle (2pt);
   \pgfusepath{fill}
\end{tikzpicture}
\caption{Counterexample for $\polynomials_1(\setEdge)\times
\polynomials_1(\setEdge)^2\times\polynomialsp_1(\setNode)$: values of
$u_h$ for $\tau=3$.}
\label{fig:values_q1_q1_3}
\end{figure}
Again, we assume \eqref{eq:counterex_assumptions}, $g_\textup D$ is given by
\eqref{eq:gD_Q1_Q1_P1} and we set $g_{{\textup D},h}=g_\textup
D|_{\Gamma_\textup D}^{}$. The solution $u_h$ is again continuous, i.e.,
$u_h\in U_h\cap C(\overline{\edge_1\cup\edge_2\cup\edge_3})$, and, at the
vertices of the hypergraph, it attains the values depicted in
Fig.~\ref{fig:values_q1_q1_3}. An explicit representation of the solution 
$\vec q_h$ can be obtained by adopting the formulas for $\vec q_h$ from the
first part of this counterexample. Then one obtains
\begin{align*}
  &(\vec q_h|_{E_1}^{}\cdot\Normal_{E_1}
  +\tau(u_h-\lambda_h))|_{\node_{12}}^{}
  +(\vec q_h|_{E_2}^{}\cdot\Normal_{E_2} 
  + \tau(u_h-\lambda_h))|_{\node_{12}}^{}\\
  &=\frac12-4(\lambda_1+\lambda_2)
  +\left(1-8(\lambda_1-\lambda_2)\right)(y-y_{E_1})
  -\frac{\lambda_3+\lambda_4}4-\frac{\lambda_3-\lambda_4}2(y-y_{E_1}),\\
  &(\vec q_h|_{E_2}^{}\cdot\Normal_{E_2}
  +\tau(u_h-\lambda_h))|_{\node_{23}}^{}
  +(\vec q_h|_{E_3}^{}\cdot\Normal_{E_3} 
  + \tau(u_h-\lambda_h))|_{\node_{23}}^{}\\
  &=-\frac{\lambda_1+\lambda_2}4-\frac{\lambda_1-\lambda_2}2(y-y_{E_1})
  -4(\lambda_3+\lambda_4)-8(\lambda_3-\lambda_4)(y-y_{E_1}),
\end{align*}
where $\node_{23}$ is the common hypernode of $\edge_2$ and $\edge_3$.
Consequently, \eqref{EQ:hdg_global} holds if and only if $\lambda_1=32/255$,
$\lambda_3=-2/255$, and $\lambda_2=\lambda_4=0$. Thus, $\lambda_h$ and $u_h$
attain negative values so that the HDG method is not positivity preserving.

It was shown in Section~\ref{sec:p0_1_p0} that using (bi)linear polynomials
for $\vec Q(\edge)$ together with constants for $U(\edge)$ and $M(\node)$ leads
to a positivity preserving method if $\tau_\edge$ are chosen appropriately.
However, it is not sufficient to consider constants only for $M(\node)$, as the
following counterexample shows.

\subsection{Counterexample for $\polynomials_1(\setEdge)\times
\polynomials_1(\setEdge)^2\times\polynomialsp_0(\setNode)$}
\label{sec:counterex_Q1_Q1_P0}

Consider again a hypergraph consisting of the two unit squares depicted in
Fig.~\ref{fig:two_squares} and assume \eqref{eq:counterex_assumptions}. 
Let $g_{{\textup D},h}(x_{E_1}-1/2,\cdot)=1$ and $g_{{\textup D},h}(x,\cdot)=0$
for $x>x_{E_1}-1/2$. Consider the HDG method \eqref{EQ:hdg_scheme} with spaces
defined using $U(\edge)=\polynomials_1(\edge)$, 
$\vec Q(\edge)=\polynomials_1(\edge)^2$, and $M(\node)=\polynomialsp_0(\node)$. 
Let $\tau_{\edge_1}=\tau_{\edge_2}=:\tau$ and denote
\begin{equation*}
   \alpha=\frac{3\tau}{4\tau+6}.
\end{equation*}
Then, setting $\lambda:=\lambda_h|_{\node_{12}}^{}$, the solution of
\eqref{EQ:hdg_primary} and \eqref{EQ:hdg_flux} is given by
\begin{equation*}
  u_h(x,y)=\left\{
    \begin{array}{cl}
       \alpha(\lambda-1)(x-x_{E_1}) + \frac{1+\lambda}4\quad
       &\text{ for all } (x,y) \in\edge_1,\\[2mm]
       -\alpha\lambda(x-x_{E_2}) + \frac\lambda4\quad
       &\text{ for all } (x,y) \in\edge_2\\
    \end{array}\right.
\end{equation*}
and
\begin{align*}
   \vec q_h|_{\edge_1}^{}(x,y)&=\begin{pmatrix} 
   1-\lambda-3(1+\lambda)(x-x_{E_1}) \\ 
   3(1+\lambda)(y-y_{E_1})+12\alpha(\lambda-1)(x-x_{E_1})(y-y_{E_1})
   \end{pmatrix},\\[1mm]
   \vec q_h|_{\edge_2}^{}(x,y)&=\begin{pmatrix} 
   \lambda-3\lambda(x-x_{E_2}) \\ 
   3\lambda(y-y_{E_2})-12\alpha\lambda(x-x_{E_2})(y-y_{E_2})\end{pmatrix}.
\end{align*}
Then
\begin{align*}
  &(\vec q_h|_{E_1}^{}\cdot\Normal_{E_1}
  +\tau(u_h|_{E_1}^{}-\lambda_h))|_{\node_{12}}^{}
  +(\vec q_h|_{E_2}^{}\cdot\Normal_{E_2} 
  + \tau(u_h|_{E_2}^{}-\lambda_h))|_{\node_{12}}^{}\\
  &=-\frac12-\frac{\tau\alpha}2+\frac\tau4
  -\lambda\left(5+\frac{3\tau}2-\tau\alpha\right)
  =-\frac{\tau^2+\tau+6}{8\tau+12}-\lambda\frac{3\tau^2+29\tau+30}{4\tau+6}
\end{align*}
and hence \eqref{EQ:hdg_global} implies that $\lambda<0$. Consequently, $u_h$
attains negative values as well so that the HDG method is not positivity
preserving.

The following counterexample demonstrates that, in contrast to the simplicial
case, the lowest-order Raviart-Thomas space on rectangles does not lead to
a positivity preserving method when $\tau_E\equiv0$.

\subsection{Counterexample for $\polynomialsp_0(\setEdge)\times
\polynomialsRT_0(\setEdge)\times\polynomialsp_0(\setNode)$ with
$\tau_E\equiv0$}
\label{sec:counter_RT}

Consider a hypergraph consisting of the two unit squares depicted in
Fig.~\ref{fig:two_squares} and assume \eqref{eq:counterex_assumptions}.
Let $g_{{\textup D},h}(x_{E_1}-1/2,\cdot)=10$ and $g_{{\textup
D},h}(x,\cdot)=0$ for $x>x_{E_1}-1/2$. It can be easily verified that the
solution of the HDG method \eqref{EQ:hdg_scheme} with $\tau_E\equiv0$ and with
spaces defined using $U(\edge)=\polynomialsp_0(\edge)$, $\vec
Q(\edge)=\polynomialsRT_0(\edge)$ (rectangular version), and
$M(\node)=\polynomialsp_0(\node)$ is given by
\begin{align*}
   &u_h|_{\edge_1}^{}=\frac94,\qquad u_h|_{\edge_2}^{}=-\frac14,\qquad
   \lambda_h|_{\node_{12}}^{}=-1,\\
   &\vec q_h|_{\edge_1}^{}(x,y)=\begin{pmatrix} 
   11-27(x-x_{E_1}) \\ 27(y-y_{E_1}) \end{pmatrix},\quad
   \vec q_h|_{\edge_2}^{}(x,y)=\begin{pmatrix} 
   -1+3(x-x_{E_2}) \\ -3(y-y_{E_2})\end{pmatrix}.
\end{align*}
Thus, the HDG method is not positivity preserving.

% ---------------------------------------------------------------------------------------------------
\section{Numerical experiments}\label{SEC:numerics}
% ---------------------------------------------------------------------------------------------------
% 
We use HyperHDG \cite{RuppGK22,HyperHDGgithub} for the numerical examples concerning graphs and rectangular elements and NGsolve for the numerical examples concerning simplices.

\subsection{Experiments on the unit hypercube}
In the first set of experiments, we consider the unit hypercube $(0,1)^d$ in
$d$ dimensions as the domain which can be viewed as the only hyperedge $\edge$
of $\setEdge$, i.e., $| \setEdge | = 1$ and $|\partial \edge| = 2d$ . We set $\kappa = 1$ and $f = 0$.

\subsubsection{The lowest-order case: $U(\edge)=\polynomialsp_0(\edge)$, $\vec Q(\edge)=\polynomialsp_0(\edge)^d$ and $M(\node)=\polynomialsp_0(\node)$}
Let us define the Dirichlet boundary $\Gamma_\textup D \subset \partial \edge$
as the part of the boundary of the hypercube that does not intersect the hyperplane $x_1 = 1$ and set $u_\textup D = 1$ if $x_1 = 0$ and $u_\textup D = 1$ otherwise.

Since $\vec Q(E) = P_0(\edge)^d$, we expect that the solution remains positive
for any $\tau > 0$, and that the problem is ill-posed for $\tau = 0$. Indeed,
the local solver fails for $\tau = 0$, and for $\tau > 0$, we receive the
graph in Fig.~\ref{fig:lowest_order} for the (constant) value of $\lambda_h$ at $x_1 = 1$. 
\begin{figure}[ht]
 \includegraphics[width=.9\textwidth]{pictures/pos_tau_const}
 \caption{The value of $\lambda_h$ at $x_1 = 1$ (ordinate) for $\tau > 0$
(abscissa) in the case $U(\edge)=\polynomialsp_0(\edge)$, $\vec
Q(\edge)=\polynomialsp_0(\edge)^d$ and
$M(\node)=\polynomialsp_0(\node)$.}\label{fig:lowest_order}
\end{figure}

\subsubsection{Multilinear flux approximation: $U(\edge)=\polynomialsp_0(\edge)$,
$\vec Q(\edge)=\polynomials_1(\edge)^d$ and $M(\node)=\polynomialsp_0(\node)$}
We consider the same setting as in the previous example. We expect that the solution remains positive for sufficiently large $\tau$ (for $\tau \ge \tfrac{2 \kappa}h$ if $d=2$), and that the problem is well-posed for $\tau = 0$. Indeed, the local solver succeeds in that case, and we receive the graph in Fig.~\ref{fig:multilinear_flux} for the (constant) value of $\lambda_h$ at $x_1 = 1$. Additionally, we deduce that our theoretical bound $\tau \ge 2$ for $d=2$ is sharp, and suspect that for dimensions $d \ge 2$, we need to ensure $\tau \ge 6$ to guarantee positivity.
\begin{figure}[ht]
 \includegraphics[width=.9\textwidth]{pictures/pos_tau_linq}
 \caption{The value of $\lambda_h$ at $x_1 = 1$ (ordinate) for $\tau > 0$
(abscissa) in the case $U(\edge)=\polynomialsp_0(\edge)$, $\vec
Q(\edge)=\polynomials_1(\edge)^d$ and
$M(\node)=\polynomialsp_0(\node)$.}\label{fig:multilinear_flux}
\end{figure}

\subsubsection{Multilinear primal approximation: $U(\edge)=\polynomials_1(\edge)$, $\vec Q(\edge)=\polynomialsp_0(\edge)^d$ and $M(\node)=\polynomials_1(\node)$}
Selecting multilinear functions for the primal and skeletal unknowns yields a methods that requires $\tau > 0$. Conversely to the other cases, we observe in Fig.~\ref{fig:multilinear_primal} that the mean of $\lambda_h$ on the Neumann boundary remains nonnegative for small $\tau$ and decreases as $\tau$ increases.
\begin{figure}[ht]
 \includegraphics[width=.9\textwidth]{pictures/pos_tau_linu}
 \caption{The mean of $\lambda_h$ at $x_1 = 1$ (ordinate) for $\tau > 0$
(abscissa) in the case $U(\edge)=\polynomials_1(\edge)$, $\vec
Q(\edge)=\polynomialsp_0(\edge)^d$ and
$M(\node)=\polynomials_1(\node)$.}\label{fig:multilinear_primal}
\end{figure}

\subsection{Raviart-Thomas approximation on simplices:
$U(\edge)=\polynomialsp_0(\edge)$, $\vec Q(\edge)=\polynomialsRT_0(\edge)$ and
$M(\node)=\polynomialsp_0(\node)$}
As a reference domain we consider the unit hypercube $\hat \Omega := (0,1)^d$ and a structured 
triangulation with $8$ elements for $d = 2$ and $48$ elements for $d = 3$. Defining the mapping 
\begin{align*}
   \phi[\theta]: \hat \Omega \to \Omega, \quad (\hat x, \hat y) \mapsto (x,y) =
(\hat x, \hat y + \theta \hat x) \qquad \text{for } d = 2, \\
   \phi[\theta]: \hat \Omega \to \Omega, \quad (\hat x, \hat y, \hat z) \mapsto
(x,y,z) = (\hat x, \hat y + \theta \hat x, \hat z + \theta \hat x) \qquad \text{for } d = 3,
\end{align*}
with $\theta \in [0,2]$, we then set $\Omega = \phi[\theta](\hat \Omega)$. In
Fig.~\ref{fig:triangulation} we have depicted the reference configuration
($\theta = 0$) and the mapped domain for $\theta = 1$ in the two-dimensional
case. On the reference domain we set the Dirichlet boundary to $\hat \Gamma_D =
\partial \hat \Omega \setminus  \hat \Gamma_N$ where $\hat \Gamma_N := \{1\}
\times (0,1)$ for $d=2$ and $\hat \Gamma_N := \{1\} \times (0,1)^2$ for $d=3$,
and set $\Gamma_D = \phi[\theta](\hat \Gamma_D)$. We consider a constant
boundary value $u_D = 1$ on $\phi[\theta](\{0\} \times (0,1))$ for $d=2$ and on
$\phi[\theta](\{0\} \times (0,1)^2)$ for $d=3$, and set $u_D = 0$ elsewhere.
\begin{figure}
\begin{center}
\begin{tikzpicture}[scale = 3]
\draw (0.5,0) -- (0.5,1); % Vertical line
\draw (0,0.5) -- (1,0.5); % Horizontal line

\coordinate (A) at (0,1);
\coordinate (B) at (0.5,0.5);
\coordinate (C) at (0.5,1);

\pic["$\alpha$", draw, fill=yellow,fill opacity=1.0,angle radius=8mm]
      {angle=C--B--A};

\draw (0,0.5) -- (0.5,0); 
\draw[draw=orange] (0.5,1) --node[below]{$\hat N$~~~~}  (1,0.5); 
\draw (0.5,0.5) -- (1,0); 
\draw (0,1) -- (0.5,0.5); 
\node[below] at (0.5,0) {$\hat \Omega$};

\draw[thick] (0,0) -- (1,0) -- (1,1) -- (0,1) -- cycle;

\draw[-{Latex[length=3mm]}] (1.2,0.5) to[bend left=20] node[above]{$\phi[1]$} (1.8,0.5);

\begin{scope}[xshift=2cm, yshift=-0.5cm]
   \draw (0.5,0.5) -- (0.5,1.5); % Vertical line
   \draw (0,0.5) -- (1,1.5); % Horizontal line

   \coordinate (A) at (0,1);
   \coordinate (B) at (0.5,1);
   \coordinate (C) at (0.5,1.5);

   \pic["$\alpha$", draw, fill=yellow,fill opacity=1.0,angle radius=8mm]
         {angle=C--B--A};

   \draw (0,0.5) -- (0.5,0.5); 
   \draw (0,1) --  (0.5,1); 
   \draw (0.5,1) -- (1,1); 
   \draw[draw=orange] (0.5,1.5) --node[below]{$N$} (1,1.5); 

   \node[] at (0.8,0.2) {$\Omega = \phi[1](\hat \Omega)$};
   \draw[thick] (0,0) -- (1,1) -- (1,2) -- (0,1) -- cycle;
\end{scope}
\end{tikzpicture}
\end{center}
\caption{The reference configuration on $\hat \Omega$ and the mapped triangulation and domain for $\theta = 1$.}\label{fig:triangulation}
\end{figure}

To motivate the transformation via $\phi[\theta]$ (for simplicity discussed
only for $d=2$), note that for example the angle $\alpha$ given in
Fig.~\ref{fig:triangulation} at the node in the center increases for an
increasing value of $\theta$. We observe that for $\theta = 1$ the angle is
$\alpha = \pi/2$ and gets bigger for $\theta > 1$. Correspondingly, condition
\eqref{IT:cond_angle} is fulfilled for all $\theta \in [0,1]$ and violated for
$\theta > 1$, which should have a direct impact on the positivity. To validate
the theory, the orange curve in the upper plot of Fig.~\ref{fig:trafoRT} shows
the value of $\lambda_h$ on the orange marked hypernode $N = \phi[\theta](\hat
N)$ where $\hat N$ is the reference hypernode between the points $(0.5,1)$ and
$(1,0.5)$. Similarly, the blue curve in the upper plot of
Fig.~\ref{fig:trafoRT} shows the value of $\lambda_h$ on a corresponding
hypernode for $d=3$. As predicted by the theory, the value of $\lambda_h$
gives negative values for $\theta > 1$ but is strictly positive for $\theta \le
1$ and $d = 2$. Since in three dimensions the negative value is not so clearly
visible, the lower plot shows the values for $\theta \in [0.8,1.2]$.

\pgfplotstableread{./pictures/data/ngs_trafoRT_2.data} \trafoRT
\pgfplotstableread{./pictures/data/ngs_trafoRT_3.data} \trafoRTthree

\begin{figure}[ht]
\begin{center}   
\begin{tikzpicture}
   \begin{axis}[width=\textwidth,height=0.3\textwidth, xmin = 0, xmax = 2, ymin=-0.1, ymax=0.15, ytick={-0.1,-0.05,0.0,0.05,0.1,0.15},y tick label style={
      /pgf/number format/.cd,
      fixed,
      fixed zerofill,
      precision=2,
      /tikz/.cd
  }]
    \addplot[thick, color=orange] table[x=0, y=1]{\trafoRT}; 
    \addlegendentry{$\lambda_h|_N^{} (d = 2)$}
    \addplot[thick, color=blue] table[x=0, y=1]{\trafoRTthree}; 
    \addlegendentry{$\lambda_h|_N^{} (d = 3)$}
    \draw[dotted] ({rel axis cs:1,0}|-{axis cs:2,0}) -- ({rel axis cs:0,0}|-{axis cs:0,0});
   \end{axis}
  \end{tikzpicture}
  \begin{tikzpicture}
   \begin{axis}[width=\textwidth,height=0.3\textwidth, xmin = 0.8, xmax = 1.2, ymin=-0.005, ymax=0.005,y tick label style={
      /pgf/number format/.cd,
      % fixed,
      fixed zerofill,
      precision=2,
      /tikz/.cd
  }]
   %  \addplot[thick, color=orange] table[x=0, y=1]{\trafoRT}; 
   %  \addlegendentry{$\lambda_h|_N (d = 2)$}
    \addplot[thick, color=blue] table[x=0, y=1]{\trafoRTthree}; 
    \addlegendentry{$\lambda_h|_N^{} (d = 3)$}
    \draw[dotted] ({rel axis cs:1,0}|-{axis cs:2,0}) -- ({rel axis cs:0,0}|-{axis cs:0,0});
   \end{axis}
  \end{tikzpicture}
\end{center}
   \caption{The value of $\lambda_h$ on the hypernode $N$ for varying values of
$\theta$ for $d=2$ (top) and $d=3$ (top and bottom).}\label{fig:trafoRT}
  \end{figure}

\subsection{The lowest-order case on quadrilaterals:
  $U(\edge)=\polynomialsp_0(\edge)$, $\vec Q(\edge)=\polynomialsp_0(\edge)^d$ and
  $M(\node)=\polynomialsp_0(\node)$}
  We use the same setting as in the previous example but now consider a quadrilateral mesh. Let $d = 2$ and set $\Omega = \phi[1.5](\hat  \Omega)$. We consider a structured mesh into $10 \times 10$ quadrilaterals on $\hat \Omega$ which are mapped accordingly. Note that we are in the setting where condition \eqref{IT:cond_angle} is not fulfilled, i.e. we demand a stabilization according to condition \eqref{EQ:tau_cond}. It can be easily seen that for a structured mesh we have for all elements $E$ with $\kappa = 1$ that 
  \begin{align*}
   |\partial E| = 0.2 + 2 \sqrt{0.1^2 + 0.05^2}
   \quad \text{and} \quad 
   \int_E \kappa^{-1} = |E| = 0.1^2,
  \end{align*}
  thus \eqref{EQ:tau_cond} reads as $\tau_E \ge 56.05.$   In Figure \ref{fig:trafoLO} we have plotted the value of the Lagrange multiplier $\lambda_h$ on the most right vertical hypernode $N$ for varying values $\tau_E \in [1,100]$. As we can see, the value of $\lambda_h$ is negative for approximately $\tau_E < 30$ and positive otherwise. This is in accordance with the theoretical prediction by \eqref{EQ:tau_cond} and validates the theory. 

  \pgfplotstableread{./pictures/data/ngs_trafoLO_2.data} \trafoLO

\begin{figure}[ht]
   \begin{center}   
   \begin{tikzpicture}
      \begin{axis}[width=\textwidth,height=0.3\textwidth, xmin = 1, xmax = 100, 
         y tick label style={
         /pgf/number format/.cd,
         fixed,
         fixed zerofill,
         precision=2,
         /tikz/.cd
     }]
       \addplot[thick, color=orange] table[x=0, y=1]{\trafoLO}; 
       \addlegendentry{$\lambda_h|_N (d = 2)$}
       \draw[dotted] ({rel axis cs:1,0}|-{axis cs:2,0}) -- ({rel axis cs:0,0}|-{axis cs:0,0});
      \end{axis}
     \end{tikzpicture}
   \end{center}
      \caption{The value of $\lambda_h$ on the hypernode $N$ for varying values of $\theta$ for $d=2$.}\label{fig:trafoLO}
     \end{figure}

\section{Conclusion}
We have shown that only a few low-order HDG methods on hypergraphs 
are positivity preserving. For nonobtuse simplicial meshes, this was proved for
the lowest-order Raviart-Thomas discretization with the stabilization parameter
$\tau=0$ and for piecewise constant approximations of all variables if
$\tau=\mathcal{O}(1)$. For general meshes (also nonsimplicial ones),
positivity preservation for piecewise constant approximations can be restored if $\tau=\mathcal{O}(1/h)$ is selected. This choice
of $\tau$ leads to positivity preservation also for a few other discretizations
as the present paper shows. However, for most HDG methods, the positivity
can be violated for any choice of $\tau$.

This work represents a first step towards a better understanding of the positivity preservation of the HDG method. Future work will focus on developing a general theory for the positivity preservation of the HDG method on general meshes and the development of high-order, positivity preserving HDG methods. While there is only little hope to find off-the-shelf high order HDG methods that are positivity preserving, we hope that our approach can be combined with flux-corrected transport (FCT) or monolithic convex limiting (MCL) approaches (see \cite{KuzminH24} for an overview) to obtain high order HDG methods that are positivity preserving. Since FCT and MCL need (at least) a consistent low-order scheme, our lowest order methods provide a cornerstone in the development of such methods.

\section*{Acknowledgment}

The work of Petr Knobloch was partly supported through the grant No. 22-01591S
of the Czech Science Foundation.
A.\ Rupp has been supported by the Research Council of Finland's decision numbers 350101, 354489, 359633, 358944, and Business Finland's project number 539/31/2023.

% \bibliographystyle{ARalpha}
% \bibliography{hdg_positive}
\printbibliography
\end{document}